\documentclass{amsart}

\usepackage[OT2,T1]{fontenc}
\usepackage[utf8]{inputenc}
\usepackage{tilmansdef}
\usepackage{tikz}
\usepackage{tikz-cd}
\usepackage{bm}
\usepackage[colorlinks=true,linkcolor=blue,citecolor=blue]{hyperref} 
\usetikzlibrary{matrix,arrows,decorations}

\usepackage{mathtools,stmaryrd}
\usepackage[backend=bibtex,style=alphabetic,doi=false,url=false,isbn=false]{biblatex}

\DeclareFontFamily{OT1}{pzc}{}
\DeclareFontShape{OT1}{pzc}{m}{it}{<-> s * [1.10] pzcmi7t}{}
\DeclareMathAlphabet{\mathpzc}{OT1}{pzc}{m}{it}

\newcommand{\IHom}{\underline{\Hom}}
\newcommand{\FgpsV}{\operatorname{Fgps}^V_{W(k)}} 
\newcommand{\FHopfc}{\operatorname{FHopf}^c_{k}} 
\newcommand{\FHopfe}{\operatorname{FHopf}^e_{k}}
\newcommand{\FHopf}{\operatorname{FHopf}_{k}} 
\newcommand{\FrM}{\operatorname{Fr}} 
\newcommand{\WtH}{\underline{\operatorname{Wt}}} 

\newcommand{\Sch}[1]{\operatorname{Sch}_{#1}}
\newcommand{\FSch}[1]{\operatorname{FSch}_{#1}}
\newcommand{\FGps}[1]{\operatorname{Fgps}_{#1}} 
\newcommand{\FGpsc}[1]{\operatorname{Fgps}^c_{#1}} 
\newcommand{\FGpse}[1]{\operatorname{Fgps}^{e}_{#1}} 
\newcommand{\AbSch}[1]{\operatorname{AbSch}_{#1}} 
\newcommand{\AbSchu}[1]{\operatorname{AbSch}^u_{#1}} 
\newcommand{\AbSchm}[1]{\operatorname{AbSch}^m_{#1}} 
\newcommand{\Hopf}[1]{\operatorname{Hopf}_{#1}}
\newcommand{\Hopfm}[1]{\operatorname{Hopf}^m_{#1}}
\newcommand{\Hopfu}[1]{\operatorname{Hopf}^u_{#1}}

\newcommand{\DMod}[1]{\operatorname{Dmod}_{#1}}
\newcommand{\DModV}[1]{\operatorname{Dmod}_{#1}^{V,\operatorname{nil}}}
\newcommand{\DDModF}[1]{\mathbb{D}\operatorname{mod}_{#1}^{F}}
\newcommand{\DDModFc}[1]{\mathbb{D}\operatorname{mod}_{#1}^{F,c}}
\newcommand{\DDModFet}[1]{\mathbb{D}\operatorname{mod}_{#1}^{F,et}}
\newcommand{\DDMod}[1]{\mathbb{D}\operatorname{mod}_{#1}}
\newcommand{\DieuP}{\operatorname{DieuP}}

\renewcommand{\ell}{\mathpzc{l}}

\DeclareMathOperator{\fin}{fin}

\newcommand{\modtensor}[2]{\rtimes}
\newcommand{\modcotensor}[2]{\hom}
\newcommand{\Spf}[1]{\operatorname{Spf}{#1}}
\newcommand{\Reg}[1]{\mathcal O_{#1}}

\newcommand{\FAlg}[1]{\operatorname{FAlg}_{#1}}
\DeclareMathOperator{\alg}{alg}
\DeclareMathOperator{\Cof}{Cof}
\DeclareMathOperator{\Bil}{Bil}
\DeclareMathOperator{\Vect}{Vect}

\DeclareMathOperator{\Frac}{Frac}
\DeclareMathOperator{\nil}{nil}

\renewcommand{\Coalg}[1]{\operatorname{Coalg}_{#1}}
\newcommand{\FCoalg}[1]{\operatorname{FCoalg}_{#1}}

\newcommand{\DieuFor}{\mathbb D^{\operatorname{f}}}

\title{Tensor products of affine and formal abelian groups}
\author{Tilman Bauer}
\address{Department of mathematics, Kungliga Tekniska H\"ogskolan\\
Lindstedtsv\"agen 25, 10044 Stockholm, Sweden}
\email{tilmanb@kth.se}
\author{Magnus Carlson}
\address{Department of mathematics, Kungliga Tekniska H\"ogskolan\\
Lindstedtsv\"agen 25, 10044 Stockholm, Sweden}
\email{macarlso@kth.se}
\date\today

\begin{document}
\maketitle
\begin{abstract}
In this paper we study tensor products of affine abelian group schemes over a perfect field $k.$ We first prove that the tensor product $G_1 \otimes G_2$ of two affine abelian group schemes $G_1,G_2$ over a perfect field $k$ exists. We then describe the multiplicative and unipotent part of the group scheme $G_1 \otimes G_2$. The multiplicative part is described in terms of Galois modules over the absolute Galois group of $k.$ We describe the unipotent part of $G_1 \otimes G_2$ explicitly, using Dieudonné theory in positive characteristic. We relate these constructions to previously studied tensor products of formal group schemes.
\end{abstract}

\section{Introduction} \label{sec:intro}

Let $\C$ be a category with finite products. The category $\Ab(\C)$ of abelian group objects in $\C$ consists of objects $A \in \C$ together with a lift of the Yoneda functor $\Hom_\C(-,A)\colon \C^\op \to \Set$ to the category $\Ab$ of abelian groups. Alternatively, $\Ab(\C)$ consists of objects $A \in \C$ with an abelian group structure $\mu\colon A \times A \to A$ with unit $\eta\colon I \to A$, where $I$ is the terminal object in $\C$, and morphisms compatible with this abelian groups structure.

On this additive category $\Ab(\C)$, it makes sense to talk about bilinear maps:

\begin{defn}
Let $\C$ be as above, and $A$, $A'$, $B \in \Ab(\C)$. Then a morphism $b\in \Hom_\C(A \times A',B)$ is called \emph{bilinear} if the induced map
\[
\Hom_\C(-,A) \times \Hom_\C(-,A') \to \Hom_\C(-,B)
\]
is a bilinear natural transformation of abelian groups. Denote by $\Bil(A,A';B)$ the set of such bilinear morphisms.

A (necessarily unique) object $A \otimes A'$ together with a bilinear morphism $a\colon A \times A \to A \otimes A'$ in $\C$ is called a \emph{tensor product} if the natural transformation
\[
\Hom_{\Ab(\C)}(A \otimes A',-) \to \Bil(A,A';-); \quad f \mapsto f \circ a, 
\]
is a natural isomorphism. 
\end{defn}

Goerss showed:

\begin{thm}[{\cite[Proposition ~5.5]{goerss:hopf-rings}}] \label{thm:goerss}
Suppose $\C$ has finite products, $\Ab(\C)$ has coequalizers and the forgetful functor $\Ab(\C) \to \C$ has a left adjoint. Then tensor products exist in $\Ab(\C)$.\qed
\end{thm}

In this paper, we study tensor products in the category $\AbSch{k}$ of abelian group objects in affine schemes over a field $k$, that is, abelian affine group schemes, or, in short, affine groups. The category of affine groups is anti-equivalent to the category $\Hopf{k}$ of abelian (i.~e. bicommutative) Hopf algebras over $k$.

Most of our results are known for finite affine groups, that is, groups $X=\Spec{H}$ where $H$ is a finite-dimensional abelian Hopf algebra over $k$ \cite{demazure-gabriel:groupes-algebriques-1,fontaine:groupes-divisibles}. The group schemes considered here are expressly not finite.

We first show that Thm.~\ref{thm:goerss} is applicable in the cases of interest:
\begin{thm}\label{thm:existence}
Tensor products exist in $\Hopf{k}$. If $k$ is perfect, then tensor products also exist in $\AbSch{k}$.
\end{thm}

We will now briefly recall the classification of affine groups, formal groups, and finite group schemes. An affine formal scheme in the sense of Fontaine \cite{fontaine:groupes-divisibles} is an ind-representable functor $X\colon \alg_k \to \Set$ from finite-dimensional $k$-algebras to sets. The category $\FGps{k}$ of (affine, commutative) formal groups over $k$ consists of the abelian group objects in the category $\FSch{k}$ of formal schemes, i.~e. functors $G\colon \alg_k \to \Ab$ whose underlying functor to sets is a formal scheme. The category $\FGps{k}$ of formal groups is anti-equivalent to the category $\FHopf$ of formal Hopf algebras, i.~e. complete pro-(finite dimensional) $k$-algebras $H$ with a comultiplication $H \to H \hat\otimes_k H$ and an antipode. 

The category $\AbSch{k}$ of affine groups is anti-equivalent to the category $\FGps{k}$ of formal groups by Cartier duality \cite[\textsection I.5]{fontaine:groupes-divisibles}. Explicitly, we have the following diagram of anti-equivalences:
\begin{equation}\label{eq:antiequivalences}
\begin{tikzcd}[column sep=large]
\AbSch{k} \arrow[d,bend left=20,"\Reg{}"] \arrow[r,bend left=5,"(-)^*"] & \FGps{k} \arrow[d,bend left=20,"\Reg{}"] \arrow[l,bend left=5,"(-)^*"]\\
\Hopf{k} \arrow[u,bend left=20,"\Spec{}"]\arrow[r,bend left=5,"{\Hom_k(-,k)}"] & \FHopf \arrow[u,bend left=20,"\Spf{}"] \arrow[l,bend left=5,"{\Hom^c_k(-,k)}"]
\end{tikzcd}
\end{equation}
Here $\Hom^c_k(H,k)$ denotes the continuous $k$-linear dual. The category of finite group schemes is, in a certain sense, the intersection of $\AbSch{k}$ and $\FGps{k}$, and thus Cartier duality gives an anti-auto-equivalence $G \mapsto G^*$ of finite affine groups, which can be thought of as the internal homomorphism object $G^* = \IHom(G,\mathbb G_m)$ into the multiplicative group $\mathbb G_m = \Spec k[\Z]$.

The reason for restricting attention to perfect fields in Thm.~\ref{thm:existence} is that the category of affine groups over them is a product of the full subcategories of \emph{multiplicative} and of \emph{unipotent} groups \cite[\textsection I.7]{fontaine:groupes-divisibles}.

 An affine group is called multiplicative if it is isomorphic to a group ring $k[G]$ for some abelian group $G$, possibly after base change to the separable (and hence algebraic) closure $\overline k$ of $k$. It is unipotent if it has no multiplicative subgroups, which is equivalent to its Hopf algebra $H$ being \emph{conilpotent}: for every element $x \in H$, there is an $n \geq 0$ such that the $n$th iterated reduced comultiplication $\psi^{(n-1)}\colon \tilde H \to (\tilde H)^{\otimes n}$ on the augmentation coideal $\tilde H = H/k$ is zero. The corresponding splitting on the formal group side is into \emph{\'etale} formal groups and \emph{connected} formal groups. A formal group $G$ is \'etale if $\Reg{G}$ is a (possibly infinite) product of finite separable field extensions of $k$, and it is called connected if $\Reg{G}$ is local, or, equivalently, if $G(k') = 0$ for all finite field extensions $k'$ of $k$. The anti-equivalences of \eqref{eq:antiequivalences} thus respect these splittings into full subcategories:

\begin{equation}\label{eq:antiequivalenceswithsplitting}
\begin{tikzcd}[column sep=small]
\AbSch{k} \ar[r,"\cong"] & \AbSchm{k} \times \AbSchu{k} \arrow[d,bend left=20,"\Reg{}"] \arrow[rr,bend left=5,"(-)^*"] && \FGpse{k} \times \FGpsc{k} \arrow[d,bend left=20,"\Reg{}"] \arrow[ll,bend left=5,"(-)^*"] \arrow[r,"\cong"] & \FGps{k}\\
\Hopf{k} \ar[r,"\cong"] & \Hopfm{k} \times \Hopfu{k} \arrow[u,bend left=20,"\Spec{}"]\arrow[rr,bend left=5,"{\Hom_k(-,k)}"] && \FHopfe \times \FHopfc \arrow[u,bend left=20,"\Spf{}"] \arrow[ll,bend left=5,"{\Hom^c_k(-,k)}"] \arrow[r,"\cong"] & \FHopf.
\end{tikzcd}
\end{equation}

Finite group schemes, which are both affine and formal, thus split into four types: multiplicative-\'etale, multiplicative-connected, unipotent-\'etale, and unipotent-connected, but such a fine splitting does not generalize to infinite groups.

Theorem~\ref{thm:existence} does not give us an explicit way to compute tensor products. The main part of this paper deals with this. In order to do this, we use alternative descriptions of the above categories:

\begin{thm}[{\cite[\textsection I.7]{fontaine:groupes-divisibles}}] \label{thm:multiplicativegroupclassification}
Let $k$ be a field with absolute Galois group $\Gamma$. Then the category $\FGpse{k}$ is equivalent to the category $\Mod_{\Gamma}$ of abelian groups with a discrete $\Gamma$-action.
Concretely, the equivalence is given by
\[
\begin{tikzcd}[column sep=5cm]
\FGpse{k} \arrow[r,bend left=5,"{G \mapsto \colim_{k \subseteq k' \subseteq \bar k} G(k')}"] & \Mod_{\Gamma}. \arrow[l,bend left=5,"{\Spf\map^\Gamma(M,\bar k) \mapsfrom M}"]
\end{tikzcd}
\]
Similarily, the category of multiplicative Hopf algebras is equivalent to $\Mod_{\Gamma}$ by
\[
\begin{tikzcd}[column sep=5cm]
\Hopfm{k} \arrow[r,bend left=5,"{\overline\Gr\colon H \mapsto \Gr(H \otimes_k \bar k)}"] & \Mod_{\Gamma}, \arrow[l,bend left=5,"{\bar k[M]^{\Gamma} \mapsfrom M}"]
\end{tikzcd}
\]
where $\Gr$ denotes the functor of grouplike elements of a Hopf algebra.
\end{thm}

In characteristic $0$, any unipotent Hopf algebra $H$ is generated by its primitives $PH$ and isomorphic to $\Sym(PH)$ by \cite{milnor-moore:hopf}; in particular, the functor $P$ is an equivalence of categories with the category $\Vect_k$ of vector spaces. We prove:

\begin{thm}\label{thm:Hopftensorchar0}
Let $k$ be a field of characteristic $0$ with absolute Galois group $\Gamma$. Then under the equivalence of categories $(\overline\Gr,P)\colon \Hopf{k} \xrightarrow{\sim} \Mod_\Gamma \times \Vect_k$, the tensor product is given by
\[
(M_1,V_1) \otimes (M_2,V_2) = \left(M_1 \otimes M_2,(M_1 \otimes \bar k)^{\Gamma} \otimes_k V_2 \oplus (M_2 \otimes \bar k)^{\Gamma} \otimes_k V_1 \oplus V_1 \otimes_k V_2\right)
\]
with unit $(\Z,0)$.
\end{thm}

\begin{thm}\label{thm:affinetensorchar0}
Let $k$ be a field of characteristic $0$ with absolute Galois group $\Gamma$. Then under the contravariant equivalence of categories $\AbSch{k} \simeq \Mod_\Gamma \times \Vect_k$, the tensor product is given by
\[
(M_1,V_1) \otimes (M_2,V_2) = (\Tor^\Z(M_1,M_2),V_1 \otimes V_2).
\]
\end{thm}

Note that this tensor product does not have a unit since $(\Z,0) \times (M,V) = 0$ for all $(M,V)$ (the pair $(\Z,0)$ corresponds to the multiplicative group $\mathbb G_m$). 

If $k$ is a perfect field of characteristic $p>0$, Dieudonné theory gives an equivalence of the category of affine groups over $k$ with modules over the ring
\[
\mathcal R = W(k)\langle F,V\rangle/(FV-p,VF-p,\phi(a)F-Fa,aV-V\phi(a)) \quad (a \in W(k)),
\]
where $W(k)$ is the ring of $p$-typical Witt vectors of $k$, $\langle F,V \rangle$ denotes the free noncommutative algebra generated by two indeterminates $F$ and $V$ (called Frobenius and Verschiebung), and $\phi\colon W(k) \to W(k)$ is the Witt vector Frobenius, a lift to $W(k)$ of the $p$th power map on $k$. We denote the subring generated by $W(k)$ and $F$ by $\mathcal F$, and the subring generated by $W(k)$ and $V$ by $\mathcal V$.

Although not intrinsically necessary, it will be convenient for the formulation of our results to restrict attention to \emph{$p$-adic affine groups} over $k$, that is, affine groups $G$ that are isomorphic to $\lim_n G/p^n$. 

The following theorem is essentially due to \cite{demazure-gabriel:groupes-algebriques-1}:
\begin{thm}
Let $\DMod{k}^p$ denote the full subcategory of left $\mathcal R$-modules consisting of those $M$ such that every $x \in M$ is contained in a $\mathcal V$-submodule of finite length. Then there is an anti-equivalence of categories
\[
D\colon \{\text{$p$-adic affine groups over $k$}\} \to  \DMod{k}^p.
\]
\end{thm}
Under this equivalence, multiplicative groups correspond to $\mathcal R$-modules where $V$ acts as an isomorphism, while unipotent groups correspond to those modules where $V$ acts nilpotently.
\begin{defn}
Let $K,\; L\in \DMod{k}$, and write $K * L$ for the $\mathcal F$-module $\Tor^{W(k)}_1(K,L)$ with the diagonal $F$-action.

Define a symmetric monoidal structure $\boxast$ on $\DMod{k}$ by
\begin{align*}
K \boxast L & \subseteq \Hom_{\mathcal F}(\mathcal R,K * L);\\
K \boxast L & = \Biggl\{f\colon \mathcal R \to K * L \Biggm| \begin{array}{c}(1 * F)f(Vr) = (V * 1)f(r)\\ (F * 1)f(Vr) = (1 * V)f(r) \end{array}\Biggr\}.
\end{align*}
We let $ K \boxast^u L \subset K \boxast L$ be the maximal unipotent submodule of $K \boxast L,$ i.e. the submodule consisting of those $x \in K \boxast L$ such that $V^nx = 0$ for some $n>0.$ It is clear that $\boxast^u$ also defines a symmetric monoidal structure on $\DMod{k}.$ 
\end{defn}
We will not state the full formula for the tensor product of two affine group schemes in all its intricate glory in this introduction (Cor.~\ref{cor:tensorcharp} for the impatient reader). We will content ourselves with giving some special cases of the formula for $G_1 \otimes G_2.$ 

Recall that if $G$ is a finite type group scheme, then $\pi_0(G)$ is the group scheme such that $\Reg{\pi_0(G)}$ is the maximal étale subalgebra of $\Reg{G}.$ When $G$ is not of finite type, it is still the filtered limit of its finite type quotient groups $G'$, and we define $\hat\pi_0(G) = \lim \pi_0(G')$ to be the corresponding pro-\'etale group. 

\begin{prop}
Let $k$ be perfect of arbitrary characteristic, and let $G_1$, $G_2$ be two affine groups over $k$ with $G_2$ of multiplicative type, i.~e. $G_2 \cong \Spec \bar k[M]^\Gamma$ for some $M \in \Mod_\Gamma$. Then $G_1 \otimes G_2$ is multiplicative, and
\[
\Reg{G_1} \boxast \Reg{G_2} \cong \bar k[\Hom^c(\hat\pi_0(G_1)(\bar{k}),M)]^{\Gamma}
\]
where $\Hom^c$ denotes continuous homomorphisms of abelian groups into the discrete module $M$, and $\Gamma$ acts by conjugation on $\Hom^c(\hat\pi_0(G_1)(\bar{k}),M).$ 
\end{prop}
The formula for $G_1 \otimes G_2$ when both $G_1$ and $G_2$ are unipotent is quite involved. The tensor product of two unipotent group schemes does \emph{not} need to be unipotent. The following Theorem gives a formula for the unipotent part of $G_1 \otimes G_2$ (for the full formula, we again refer to Cor.~\ref{cor:tensorcharp}) 
\begin{thm}
Let $G_1,G_2$ be unipotent groups over a perfect field $k$ of positive characteristic. Then the unipotent part of $D(G_1 \otimes G_2)$ is isomorphic to $$D(G_1) \boxast^u D(G_2).$$ 
\end{thm}
In \cite{Hedayatzadeh:Exterior}, Hedayatzadeh studies tensor products of finite $p$-torsion group schemes over a perfect field $k$ of characteristic $p>0.$ He shows that the tensor product $G_1 \otimes G_2$ of two finite $p$-torsion group schemes $G_1, G_2$ exists and is a group scheme that is an inverse limit of finite $p$-torsion group schemes. A monoidal structure $\boxtimes$ is then defined on the category of Dieudonné modules, which is the same as the monoidal structure defined by Goerss in \cite{goerss:hopf-rings}. Let $D^*$ be the covariant Dieudonné functor, which takes a finite $p$-torsion group scheme $G$ to $D(G^*),$ where $G^*$ is the Cartier dual. Hedayatzadeh then shows that if $G_1$ and $G_2$ are two finite $p$-torsion group schemes, then $D^*(G_1 \otimes G_2)$ is naturally isomorphic to $ D^*(G_1) \boxtimes D^*(G_2),$ that is, $D^*$ is monoidal. Our results generalize his results on tensor products of finite $p$-torsion group schemes to all affine group schemes over a perfect field $k.$ The methods we use are different from the ones of Hedayatzadeh's, we work more in the spirit of \cite{goerss:hopf-rings}.
\section{Tensor and cotensor products} \label{sec:tensor-cotensor}

In this section, we will show that tensor products exist in the category of affine groups and in the category of abelian Hopf algebras over perfect fields $k$. 

We can apply Theorem~\ref{thm:goerss} to show the first part of Theorem~\ref{thm:existence}:

\begin{thm}\label{thm:tensorproductsofformalgroupsexist}
Let $k$ be a field. Then tensor products exist in the categories $\Hopf{k}$ of abelian Hopf algebras and $\FGps{k}$ of formal groups.
\end{thm}
\begin{proof}
The category $\Hopf{k}$ has all colimits, in particular coequalizers, and the forgetful functor $\Hopf{k} \to \Coalg{k}$ from abelian Hopf algebras to coalgebras has a left adjoint, the free abelian Hopf algebra functor. By Theorem~\ref{thm:goerss}, the tensor product exists. 

The category $\FSch{k}$ of affine formal schemes is equivalent with the category $\Coalg{k}$ of cocommutative $k$-coalgebras by the fundamental theorem on coalgebras over a field \cite{sweedler:hopf-algebras}: any coalgebra $C$ is the union of the directed set of its finite-dimensional sub-coalgebras $C_i$, and the functor $C \mapsto \Spf{\Hom(C_i,k)_i}$ gives the desired equivalence. Since abelian group objects in $\FSch{k}$ are precisely formal groups, tensor products also exist in $\FGps{k}$.
\end{proof}

\begin{remark}\label{rem:formalgroupunit}
The unit object in $\Hopf{k}$ is the free Hopf algebra on the coalgebra $k$, which is the group ring $k[\Z]$. Hence in formal groups, the unit is the constant formal group $\underline{\Z} = \Spf{k^\Z}$.
\end{remark}

Dealing with the case of affine groups is not quite as straightforward. Although the abelian group objects in $\Sch{k}$ and in $\Coalg{k}$ are in both cases the abelian Hopf algebras, the tensor products are not the same; rather, they are dual to each other. The tensor product of affine groups can be thought of as classifying \emph{cobilinear} maps of Hopf algebras.

To construct tensor products, we need to show that for a perfect field $k,$ there exists a free affine group functor, i.e a left adjoint to the forgetful functor $\AbSch{k} \rightarrow \Sch{k}$ to the category of affine schemes. To construct this functor, recall (e.g. from  \cite[Theorem 14.83]{GortzWedhorn:AG}):
\begin{prop}\label{prop:galoisdescent}
For any field $k$ with absolute Galois group $\Gamma$, extension of scalars from the category $\Alg_k$ to the category $\Alg_{\bar k,\Gamma}$ of $\bar k$-algebras with a continuous semilinear $\Gamma$-action, is an equivalence. The inverse is given by $\Gamma$-fixed points. \qed
\end{prop}

We can now show the existence of tensor products of affine groups: 
\begin{proof}[Proof of Thm.~\ref{thm:existence}]
The category $\AbSch{k}$ of affine groups has coequalizers since the category $\Hopf{k}$ has equalizers, and the existence of tensor products of affine groups follows from Theorem ~\ref{thm:goerss} if we can show that there is a cofree abelian Hopf algebra functor on commutative $k$-algebras. By \eqref{eq:antiequivalenceswithsplitting}, it is enough to construct a cofree multiplicative Hopf algebra functor and a cofree unipotent Hopf algebra functor separately.

Given a $k$-algebra $A$, the absolute Galois group $\Gamma$ acts on the group ring $\bar{k}[(A \otimes \bar{k})^*]$, and the cofree multiplicative Hopf algebra on $A$ is given by the fixed points of this action. Indeed, the category of multiplicative groups over $k$ is equivalent with the category of modules with a continuous action of $\Gamma$. This equivalence together with ~\ref{prop:galoisdescent} easily implies our statement.

The cofree unipotent Hopf algebra on a $k$-algebra $A$ was first constructed by Takeuchi \cite[Prop.~1.5.7]{takeuchi:tangent-coalgebras}. This is the maximal cocommutative sub-Hopf algebra of the cofree non-cocommutative unipotent Hopf algebra, which was later constructed in \cite{newman-radford:cofree}. The latter is, as a vector space, the tensor algebra $C(A)=\bigoplus_{n\geq0} A^{\otimes_k n}$, which obtains a comultiplication by splitting up tensors in all possible ways, and inherits a multiplication from $A$. The \emph{cocommutative} unipotent Hopf algebra, then, is the sub-Hopf algebra of symmetric tensors $\bigoplus_{n\geq 0} (A^{\otimes_k n})^{\Sigma_n}$.
\end{proof}

\begin{example}\label{ex:cofreeonk}
The cofree cocommutative Hopf algebra on the algebra $k$ is given as follows. Its multiplicative part is $\bar{k}[\bar{k}^\ast]^{\Gamma}$. In characteristic~$0$, its unipotent part is the primitively generated Hopf algebra $k[x]$, while in characteristic~$p$, it is the unique Hopf algebras structure on $k[b_{(0)},b_{(1)},\dots]/(b_{(i)}^p-b_{(i)})$ with $b_{(0)}$ primitive and the Verschiebung acting by $V(b_{(i+1)})=b_{(i)}$ for $i \geq 0$. Thus, as opposed to the case of formal schemes (Rem.~\ref{rem:formalgroupunit}), this basic free object is neither connected nor \`etale.
\end{example}

\section{Hopf algebras and formal groups in characteristic zero} \label{sec:char0}

Throughout this section, let $k$ be a field of characteristic $0$ with algebraic closure $\bar k$ and absolute Galois group $\Gamma$. The aim of this section is to describe the tensor products of (formal) group schemes over $k$ explicitly and constructively.

By \eqref{eq:antiequivalenceswithsplitting}, any Hopf algebra over $k$ splits as a product $H^m \otimes H^u$ of a Hopf algebra of multiplicative type and a unipotent Hopf algebra. The functor of primitives $P\colon \Hopfu{k} \to \Vect_k$ is an equivalence with inverse the symmetric algebra functor $\Sym$. It follows immediately that the tensor product $H_1 \boxtimes H_2$ of \emph{unipotent} Hopf algebras is isomorphic to $\Sym(PH_1 \otimes PH_2)$.

The situation for Hopf algebras of multiplicative type is a bit muddier. Let us first assume that $k$ is algebraically closed. Then by Thm.~\ref{thm:multiplicativegroupclassification}, every Hopf algebra $H$ of multiplicative type is the group ring of its grouplike elements: $H=k[\Gr(H)]$. Thus the tensor product of multiplicative Hopf algebras is given by $H_1 \boxtimes H_2 \cong k[\Gr(H_1) \otimes \Gr(H_2)]$.

If $H_1 \cong k[A]$ is of multiplicative type and $H_2 = \Sym(V)$ is unipotent then there is an isomorphism of Hopf algebras
\[
k[A] \boxtimes \Sym(V) \to \Sym(A \otimes V); \quad [a] \otimes v \mapsto a \otimes v \text{ for } a \in A,\; v \in V.
\]
This proves Thm.~\ref{thm:Hopftensorchar0} when $k=\bar k$. For a general $k$, the equivalence
\[
\Gr \times P \colon \Hopf{\bar k} \to \Ab \times \Vect_{\bar k}
\]
is $\Gamma$-equivariant and thus gives an equivalence
\[
\Hopf{\bar k,\Gamma} \to \Mod_{\Gamma} \times \Vect_{\bar k,\Gamma} \simeq_{\text{Prop.~\ref{prop:galoisdescent}}} \Mod_{\Gamma} \times \Vect_k
\]
and Thm.~\ref{thm:Hopftensorchar0} follows.

\begin{example}
Let $H_1 = \Q[x]$ be the unipotent Hopf algebra primitively generated by a variable $x$, and $H_2 = \Q[x,y]/(x^2+y^2-1)$ the Hopf algebra with comultiplication given by $x \mapsto x \otimes x - y \otimes y$ and $y \mapsto y \otimes x - x \otimes y$. Then $H_2$ is multiplicative; in fact, the map
\[
H_2 \to \Q[i][t,t^{-1}], \quad x \mapsto \frac12 (t+t^{-1}), \quad y \mapsto \frac12 (it-it^{-1})
\]
is an isomorphism of Hopf algebras between $H_2$ and $\Q[i][t^{\pm 1}]^{C_2} = \bar \Q[t^{\pm 1}]^{\Gamma_\Q}$, where $C_2$ acts by sending a Laurent polynomial $p(t)$ to $\overline{p(t^{-1})}$. This shows that the Galois module corresponding to $H_2$ is $\Z^\sigma$, the abelian group $\Z$ with $\Gamma_\Q$ acting nontrivially through its quotient $C_2$.

The above results imply:
\begin{itemize}
\item $H_1 \boxtimes H_1 \cong \Sym(k\langle x \rangle) \boxtimes \Sym(k \langle x \rangle) = \Sym(k \langle x \otimes x \rangle) \cong H_1$;
\item $H_2 \boxtimes H_2$ corresponds to the Galois module $\Z^\sigma \otimes \Z^\sigma \cong \Z$ and thus $H_2 \boxtimes H_2 \cong k[\Z]$;
\item $H_1 \boxtimes H_2$ is the unipotent Hopf algebra associated with the module $(\Z^{\sigma} \otimes \Q[i])^{C_2} = \Q[i]^{C_2} = \Q\langle i \rangle$ since $C_2$ acts by negated complex conjugation. Thus $H_1 \boxtimes H_2 \cong H_1$.
\end{itemize}
\end{example}

By the equivalence of the categories of coalgebras and formal schemes, we can rephrase Theorem ~\ref{thm:Hopftensorchar0} as follows:
\begin{corollary}\label{cor:fschtensorchar0}
There is an equivalence of symmetric monoidal categories
\[
\FGps{k} \to \Mod_{\Gamma} \times \Vect_k
\]
given by $G \mapsto \left(\Hom_{\FGps{k}}(\underline \Z,G_{\bar k}), \Hom_{\FGps{k}}(\hat{\mathbb G}_a,G)\right)$, with the monoidal structure on the target category as in Thm~\ref{thm:Hopftensorchar0}.
\end{corollary}

\section{Tensor products of multiplicative affine groups}

Given two affine groups $G_1$, $G_2$ over a field $k$, denote by $\boxast$ the operation on $k$-Hopf algebras which satisfies
\[
\Reg{G_1 \otimes G_2} \cong \Reg{G_1} \boxast \Reg{G_2}.
\]
We will now describe this operation explicitly for tensor products with at least one factor of multiplicative type, we make use of the following:

\begin{defn} \label{def:pi0}
Let $G$ be an affine group over any perfect field $k$. The \emph{pro-\'etale group of connected components} $\hat\pi_0(G)$ of $G$ is the profinite group $\lim_{G'} \pi_0(G')$, where $G'$ ranges through all quotients of $G$ of finite type, and $\pi_0(G')$ denotes the group of connected components of the algebraic group $G'$ \cite[\textsection 5.i]{milne:algebraic-groups}, i.~e. the initial \'etale group under $G'$.
\end{defn}
Note that $\hat\pi_0(G)(\bar{k}) = \lim_{i \in I } \pi_0(G_i)(\bar{k})$ is a profinite $\Gamma$-module, which is finite if $G$ is of finite type.

\begin{prop} \label{prop:multiplicativetensor}
Let $k$ be perfect of arbitrary characteristic, and let $G_1$, $G_2$ be two affine groups over $k$ with $G_2$ of multiplicative type, i.~e. $G_2 \cong \Spec \bar k[M]^\Gamma$ for some $M \in \Mod_\Gamma$. Then $G_1 \otimes G_2$ is multiplicative, and
\[
\Reg{G_1} \boxast \Reg{G_2} \cong \bar k[\Hom^c(\hat\pi_0(G_1)(\bar{k}),M)]^{\Gamma}
\]
where $\Hom^c$ denotes continuous homomorphisms of abelian groups into the discrete module $M$, and $\Gamma$ acts by conjugation on $\Hom^c(\pi_0(G_1)(\bar{k}),M).$ 
\end{prop}
\begin{proof}
Let show first that $G_1 \otimes G_2$ is multiplicative, so that there can be no unipotent part. To show this, it is enough by \cite[IV \textsection 1 n\textsuperscript{o} 2, Théorème 2.2.]{demazure-gabriel:groupes-algebriques-1} that $0 = \Hom_{\AbSch{k}}(G_1 \otimes G_2,\mathbb{G}_a) = \Hom(G_1,\underline\Hom(G_2,\mathbb{G}_a))$, where the outer Hom is of abelian-group valued functors. By \cite[Exposé XII, Lemme 4.4.1]{Gille:SGA3}, for any $k$-algebra $A,$ any group homomorphism from a multiplicative group over $\Spec A$ into $\mathbb{G}_{a,A}$ must be trivial. Thus $\underline\Hom(G_2,\mathbb{G}_a) = 0$.

What is left is to determine the multiplicative part. Let $K$ be a multiplicative group, without loss of generality of finite type. Then by \cite[Exposé VIII, Corollaire 1.5]{Gille:SGA3}, the group-valued functor $\underline{\Hom}(G_2,K)$ is isomorphic to the constant group associated to the abelian group 
\[
\Hom_{\Ab}(\overline{\Gr}(\Reg{K}),\overline{\Gr}(\Reg{G_2})) = \Hom_{\Ab}(\overline{\Gr}(\Reg{K}),M)
\]
after base change to $\bar k$. Since $\underline{\Hom}(G_2,K)$ is a fpqc sheaf, \cite[IV, \textsection 1, n.\kern-0.5bp\textsuperscript{o} 3, Lemme 3.1]{demazure-gabriel:groupes-algebriques-1} implies by descent that $\underline{\Hom}(G_2,K)$ is an \'etale scheme over $k$. Thus, 
\begin{align}
\Hom_{\AbSch{k}}(G_1 \otimes G_2,K) \cong &\Hom_{\AbSch{k}}(G_1,\underline{\Hom}(G_2,K)) \notag \\
\cong & \Hom_{\AbSch{k}}^c(\hat\pi_0(G_1),\underline\Hom(G_2,K)) \label{eq:etalehomadjunction}
\end{align}
 since the identity component of $G_1$ must be in the kernel of any morphism to an étale group. Furthermore, the category of pro-étale groups over $k$ is equivalent to the category of profinite $\Gamma$-modules by the functor taking an \'etale group to its $\bar{k}$-rational points. Thus, \eqref{eq:etalehomadjunction} is isomorphic to
 \[
 \Hom^c_{\Mod_{\Gamma}}(\hat \pi_0(G_1)(\bar{k}), \Hom_{\Ab}(\overline\Gr(\Reg{K}),M)),
 \]
 where the inner Hom carries the discrete topology. By adjunction, this is in its turn isomorphic to
 \[
 \Hom_{\Mod_{\Gamma}}(\overline\Gr(\Reg{K}), \Hom_{\Ab}^c(\hat\pi_0(G_1)(\bar{k}),M)).
 \]
By Thm.~\ref{thm:multiplicativegroupclassification}, this is just  
\[
\Hom_{\AbSch{k}}(\Spec \bar k[\Hom_{\Ab}^c(\hat \pi_0(G_{1})( \bar{k}),M)]^{\Gamma}, K),
\]
concluding the proof.
\end{proof}

\begin{example} \label{ex:tensormup}
Let $\mu_n$ be the affine group taking a $k$-algebra to its $n$th roots of unity, and assume for simplicity that $k$ is algebraically closed. Then $$\mu_n \otimes \mu_n \cong \Spec k[\Hom^c(\hat \pi_0(\mu_n)(k), \Z/n\Z)].$$ If the characteristic of $k$ does not divide $n,$ then $\mu_n \cong \Z/n\Z,$ so that $\hat \pi_0(\mu_n)(k) \cong \underline{\smash{\Z/n\Z}},$ which gives $\mu_n \otimes \mu_n \cong \underline{\smash{\Z/n\Z}}.$ However, if the characteristic of $k$ does divide $n,$ this is not true. For example, if $n = \operatorname{char}(k),$ then $$\mu_n \otimes \mu_n = 0.$$ 
\end{example}

\begin{example} \label{ex:tensorGm}
For the group $\mathbb G_m$ over a perfect field $k$, $\hat\pi_0(\mathbb G_m)(\bar k)=0$, and hence
\[
\mathbb{G}_m \otimes G \cong \Spec \bar k[\Hom^c_{\Ab}(\hat \pi_0(\mathbb{G}_m)(\bar k),\overline \Gr(\Reg{G}))]^\Gamma =0
\]
for all affine groups $G$. This shows that the tensor product of affine groups cannot have a unit. 
\end{example}

\begin{example}\label{ex:tensormpzp}
Let $k$ be an algebraically closed field of characteristic $p>0$. The constant group $\underline{\smash{\Z/p\Z}}$ is unipotent, and Prop.~\ref{prop:multiplicativetensor} gives that $$\underline{\smash{\Z/p\Z}} \otimes \mu_p \cong \Spec k[\Hom_{\Ab}(\Z/p\Z, \Z/p\Z)] \cong \mu_p.$$ Thus the tensor product of a unipotent group and a multiplicative group does not need to be trivial. 
\end{example}

\begin{lemma}\label{lemma:pi0ofmultiplicative}
Let $k$ be perfect of arbitrary characteristic and $G = \bar k[M]^\Gamma$ the multiplicative group corresponding to a $\Gamma$-module $M$. Then
\[
\hat\pi_0(G) \cong \varprojlim_{M' \subset M}  \Spec \bar k[M']^\Gamma,
\]
where $M' < M$ runs through the finitely generated submodules of torsion prime to $p$ if $\operatorname{char}(k)=p$ and through all finitely generated torsion submodules if $\operatorname{char}(k)=0$.

In particular,
\[
\pi_0(G)(\bar k) \cong \lim \Hom(M',\bar k^\times)
\]
\end{lemma}
\begin{proof}
By its definition, $\hat \pi_0(G) =\varprojlim_{M' \subset M}  \Spec \bar k[M']^\Gamma$, where $M'$ runs through all finitely generated submodules of $M$. Thus it suffices to show that
\[
\pi_0(G) \cong \bar k[M']^\Gamma,
\]
where $G$ is of finite type (i.e. $M$ is finitely generated) and $M'$ is its torsion submodule (torsion prime to $p$ if $\operatorname{char}(k)=p>0$.)
The inclusion $M' \hookrightarrow M$ induces a $\Gamma$-equivariant map $\bar k[M'] \to \bar k[M]$. One sees that it is enough to prove the theorem when $k = \bar{k}.$ By the structure of finitely generated abelian groups, it thus suffices to show that $\pi_0(\Spec k[M]) = 0$ for $M = \Z$ (and $M=\Z/n\Z$ where $n = p^k$  if $\operatorname{char}(k)=p>0$) and that $\pi_0(\Spec k[ \Z/n\Z]) = \Spec k[\Z/n\Z]$ if $ n \nmid p.$ Indeed, since $\Spec k[\Z] \cong \mathbb{G}_m,$  we have that $\pi_0(\Spec k[\Z]) \cong \pi_0(\mathbb G_m)=0$ and analogously, since $\Spec k[\Z/p^n \Z] \cong \mu_{p^n}$ we have that  $\pi_0(\Spec k[\Z/p^n\Z]) =0$ if $0 \neq \operatorname{char}(k) .$ If $n \nmid p,$ then $\Spec k[\Z/n\Z] \cong \underline{\Z/n\Z},$ the constant group scheme on $\Z/n\Z$ and it is obvious that $\pi_0(\underline{\Z/n\Z}) \cong \underline{\Z/n\Z} .$ The second part follows by noting, that if $G \cong \Spec \bar{k}[M]^\Gamma$ for $M$ an abelian group, then $$G(\bar{k}) \cong \Hom_{\Alg_{\bar{k}}}(\bar{k}[M],\bar{k}) \cong \Hom(M,\bar{k}^*).$$ 
\end{proof}

For $\Gamma$-modules $M_1$, $M_2$, denote by $M_1 * M_2$ the $\Gamma$-module $\Tor_1^\Z(M_1,M_2)$ with the ``'diagonal'' $\Gamma$-action defined as follows: if the $\Gamma$-actions on $M_i$ are given by pro-maps $\Gamma \to \End(M_i)$, where $\End(M_i) = \{\Hom(M',M_i)\}_{M' <M_i \text{ f.g.}}$, then the diagonal map
\[
\Gamma \to \End(M_1) \times \End(M_2) \xrightarrow{*} \End(M_1 * M_2)
\]
induces a continuous action.

\begin{corollary}\label{cor:tensorofmultiplicative}
Let $k$ be perfect of arbitrary characteristic, and let $G_i = \Spec \bar k[M_i]^\Gamma$ be multiplicative groups associated to $\Gamma$-modules $M_i$ ($i=1,2$). Then
\[
G_1 \otimes G_2 \cong \begin{cases} \Spec \bar k[M_1*M_2]^\Gamma; & \operatorname{char}(k)=0\\
\Spec \bar k[\Z[1/p] \otimes M_1*M_2]^\Gamma; & \operatorname{char}(k)=p>0.
\end{cases}
\]
\end{corollary}
\begin{proof}
By Proposition ~\ref{prop:multiplicativetensor} and Lemma \ref{lemma:pi0ofmultiplicative}, the first component is given by
\[
\hat \Gr(\Reg{G_1} \boxast \Reg{G_2}) \cong \Hom^c(\hat \pi_0(\Spec \bar k[M_1]^\Gamma)(\bar k),M_2) = \colim_{M'} \Hom(\Hom(M',\bar k^\times), M_2),
\]
where $M'$ runs through the finitely generated torsion submodules of $M_1$ (torsion prime to $p$ if $\operatorname{char}(k)=p$).
In characteristic $0$, the largest torsion submodule of $\bar k^\times$ is isomorphic to $\Q/\Z,$ since $\mathbb{Q} \subset k$ and any torsion point of $\bar{k}^\times$ must be a root of unity. Hence the above equals $\colim_{M'} \Hom(\Hom(M',\Q/\Z),M_2) \cong M_1*M_2$.
In characteristic $p$, $\bar{\F}_p \subset \bar{k},$ and we see that any homomorphism $M' \rightarrow \bar{k}^\times$ must factor through $\bar{\F}^\times_p.$ This latter group is isomorphic to the prime-to-$p$ torsion in $\Q/\Z.$ We thus see that in characteristic $p,$ we have that
\[
\colim_{M'} \Hom(\Hom(M',\bar{k}^\times),M_2) \cong \colim_{M'} \Hom(\Hom(M',\Q/\Z), M_2) \cong \colim_{M'} M'*M_2.
\]
The statement now follows since $\Z[1/p] \otimes (M_1*M_2) \cong \colim_{M'} (M'*M_2).$

\end{proof}

\begin{proof}[{Proof of Thm.~\ref{thm:affinetensorchar0}}]
In characteristic $0$, any unipotent \'etale group is trivial, and hence $\hat \pi_0(G)\cong \hat \pi_0(G^m)$, where $G^m$ is the multiplicative part of an arbitrary affine group $G$. Furthermore, every unipotent Hopf algebras $H$ is isomorphic to $\Sym(P(H))$. In particular, for $H=\Cof(A)$ the cofree unipotent Hopf algebra on an algebra $A$, 
\[
PH \cong \Hom_{\Hopf{k}}(\Reg{\mathbb{G}_a},\Cof(A)) \cong \Hom_{\Alg_k}(k[x],A) \cong A
\]
and hence $\Cof(A) \cong \Sym(A)$. This shows that for unipotent groups $G_1$, $G_2$,
\begin{equation}\label{eq:unipotentgrouptensorchar0}
\Reg{G_1} \boxast \Reg{G_2} \cong \Sym(P(\Reg{G_1}) \otimes P(\Reg{G_2})).
\end{equation}
The theorem now follows from Cor.~\ref{cor:tensorofmultiplicative}
.\end{proof}

\section{Smooth formal groups over a perfect field of positive characteristic} \label{sec:smoothFormal}

From now on, let $k$ be a perfect field of characteristic $p>0$ with algebraic closure $\bar k$ and absolute Galois group $\Gamma$. We will denote by $W(k)$ the ring of ($p$-typical) Witt vectors of $k$, equipped with the Verschiebung operator $V$ and the Frobenius operator $F$. As in \cite{goerss:hopf-rings}, studying the category of formal groups over fields $k$ is simplified by studying universal objects, which actually are mod-$p$ reductions of smooth formal groups defined over $W(k)$. Here the topology of $W(k)$ is captured by its structure of a pseudocompact ring.

\subsection{Pseudocompact rings and formal schemes over them}

We give a brief overview over the notion of pseudocompact rings and modules in the commutative setting. The reader may want to consult \cite[Exposé VII\textsubscript{B}]{Gille:SGA3} or \cite[Ch.1, \textsection 3]{fontaine:groupes-divisibles} for more details.

\begin{defn}
A (unital, commutative) linearly topologized ring $A$ is called \emph{pseudocompact} if its topology is complete Hausdorff and has a basis of neighborhoods of $0$ consisting of ideals $I$ such that $A/I$ has finite length as an $A$-module.

Similarly, if $A$ is a pseudocompact ring, then a \emph{pseudocompact $A$-module} $M$ is a complete Hausdorff topological $A$-module which admits a basis of neighborhoods of $0$ consisting of submodules  $M' \subset M$ such that $M/M'$ has finite length as an $A$-module.

Morphisms of pseudocompact rings and modules are by definition continuous ring and module homomorphisms, respectively.
\end{defn}

Note that any Artinian ring is trivially pseudocompact, as is any complete local Noetherian ring, such as $W(k)$.
Given two pseudocompact $A$-modules $M$ and $N$, we can form the completed tensor product $M \hat{\otimes}_{A} N$, which is the inverse limit of the tensor products $M/M' \otimes_A N/N'$ where $M' \subset M$ and $N' \subset N$ range through the open submodules of $M$ and $N$, respectively. A pseudocompact $A$-module $M$ is \emph{topologically flat} (or equivalently, projective) if $- \hat{\otimes}_A M$ is an exact functor. If it is a fortiori isomorphic to a direct product of copies of $A$, we call it \emph{topologically free}. Then $M$ is projective if and only if it is locally topologically free, in the sense that the base change $M \hat{\otimes}_A A_{\mathfrak m}$ is topologically free for every maximal open ideal $\mathfrak m \triangleleft A.$

If $A$ is a pseudocompact ring, we say that a commutative $A$-algebra $B$ is a \emph{profinite $A$-algebra} if the underlying $A$-module is pseudocompact. Denoting the category of finite length $A$-algebras by $\FAlg{A}$, a profinite $A$-algebra $B$ represents a functor $\Spf B\colon \FAlg{A} \to \Set$,
\[
\Spf B(R) = \Hom^{c}_A(B,R) \quad \text{(continuous homomorphisms)}.
\]
A functor $\FAlg{A} \to \Set$ is a \emph{formal scheme} if it is representable in this way. A formal scheme $\Spf B$ is said to be \emph{connected} if $B$ is a local $A$-algebra. The category of formal schemes has all limits. A \emph{formal group} $G$ over $A$ is an abelian group object in the category of formal schemes. We call $G$ \emph{smooth} if for any finite $A$-algebra $B$ and any square-zero ideal $I \subset B,$ the canonical map $G(B) \rightarrow G(B/I)$ is surjective. Any étale formal group is smooth, and a connected formal group is smooth if and only if its representing profinite $A$-algebra is a power series algebra \cite[\textsection I.9.6]{fontaine:groupes-divisibles}.

\subsection{Smooth formal groups with a Verschiebung lift and their indecomposables}\label{subsec:fgwithVlift}

Let $G$ be a formal group over $A=W(k)$ with representing formal Hopf algebra $\Reg{G}$. An endomorphism $V_G\colon G \to G$ is called a \emph{lift of the Verschiebung} if its base change to $k$ is the Verschiebung on $G_k.$  Note that unless $k = \mathbb{F}_p,$ the map $V_G$ is not one of formal groups over $W(k).$ Instead, denoting by $F_{W(k)}$ the Frobenius on $W(k),$ we have that $V_G(ax) = F^{-1}_{W(k)}(a) V_G(x)$ for $a \in W(k), x \in \Reg{G}.$ We say that $V_G$ is $F^{-1}_{W(k)}$-linear.

\begin{defn}
The category $\FgpsV$ is the category whose objects are pairs $(G,V_G)$ where $G$ is a connected, smooth formal group over $W(k)$ and $V_G$ is a lift of the Verschiebung. A morphism $(G_1,V_{G_1}) \rightarrow (G_2,V_{G_2})$ in $\FgpsV$ is given by a morphism of formal groups $f\colon G_1 \rightarrow G_2$ such that $fV_{G_1} = V_{G_2} f.$ 

Denote by $\mathcal{H}_V$ the full subcategory of complete $W(k)$-Hopf algebras representing objects in $\FgpsV$. 
\end{defn}

\begin{defn} \label{def:MV}
The category $\mathcal{M}_V$ is the category whose objects are topologically free $W(k)$-modules $M$ together with a continuous $F^{-1}_{W(k)}$-linear endomorphism $V_M.$  A morphism 
\[
(M_1,V_{M_1}) \rightarrow (M_2,V_{M_2})
\]
 in $\mathcal{M}_V$ is a morphism $f\colon M_1 \rightarrow M_2$ of pseudocompact modules such that $fV_{M_1} = V_{M_2}f.$
\end{defn}
Denoting by $I_G$ the augmentation ideal of $\Reg{G} \in \mathcal H_V$, the (contravariant) functor of indecomposables is defined by
\[
Q\colon \FgpsV \rightarrow \mathcal{M}_V, \quad G \mapsto I_G/\overline{I_G}^2,
\]
where $\overline{I_G}^2$ is the closure of $I_G^2$ in $\Reg{G}.$

 The following theorem is the main theorem of this section.
\begin{thm}\label{thm:indecomposableequiv}
The contravariant functor 
\[
Q : \FgpsV \rightarrow \mathcal{M}_V
\]
 is an anti-equivalence of categories.
\end{thm}

Two objects in $\FgpsV$ will figure in the proof, namely, the co-Witt vector and the finitely supported Witt vector functors, which will be introduced in Subsection~\ref{subsec:backgroundWitt}. Furthermore, the proof requires dualization of the category $\FgpsV$, and the dual category will be described explicitly in Subsection~\ref{subsec:OWkalgebras}. 

\subsection{Witt vectors and co-Witt vectors}  \label{subsec:backgroundWitt}

While our main interest is in the Witt vector and co-Witt vector constructions for $k$-algebras, we will need them briefly also for $W(k)$-algebras. Thus we give a brief review of these constructions for a general (commutative) ring in this subsection.

For a ring $A$, the ring of ($p$-typical) Witt vectors $W(A)$ has two operators $F$ and $V$, where $V(a_0,a_1,\dots) = (0,a_0,a_1,a_2,\dots)$ is the Verschiebung (shift) operator, and $F$ is the Frobenius. If $\operatorname{char}(A)=p$ then $F(a_0,a_1,\dots) = (a_0^p,a_1^p,\dots)$ and $VF=FV=p$, but in the general case only the identity $FV=p$ holds. The Frobenius $F$ is a ring map, while the Verschiebung $V$ satisfies Frobenius reciprocity: $aV(b) = V(F(a)b$. For $k=A=\F_p$, the Verschiebung is multiplication by $p$ on $W(\F_p)$, while the Frobenius is the identity. 

The truncated, or finite-length, Witt vectors $W_n(A)$ are defined similarily as sequences of length $n+1$, and the Verschiebung naturally lives as a map $V\colon W_n(A) \to W_{n+1}(A)$

Modelled by the short exact sequence $0 \to \Z_p \to \Q_p \to \Q_p/\Z_p \to 0$ for the case $A=\F_p$, we define for any ring $A$ a natural short exact sequence of abelian groups
\begin{equation}\label{eq:unipotentwittcowittsequence}
0 \to W(A) \to QW^u(A) \to CW^u(A) \to 0,
\end{equation}
where
\[
QW^u(A) = \colim (W(A) \xrightarrow{V} W(A) \xrightarrow{V} \cdots) = \{(a_i) \in A^\Z \mid a_i=0 \text{ for } i\ll 0\}
\]
is the group of \emph{unipotent bi-Witt vectors}, and
\[
CW^u(A) = QW(A)/W(A) = \colim (W_0(A) \xrightarrow{V} W_1(A) \xrightarrow{V} \cdots).
\]
is the group of unipotent co-Witt vectors. In \cite{fontaine:groupes-divisibles}, a non-``unipotent'' version $CW(A)$ of $CW^u(A)$ is constructed with
\[
CW(A) = \{ (a_i) \in A^{\Z_{\leq 0}} \mid a_i \text{ nilpotent for }i \ll 0\},
\]
and in a similar way, the group-valued functor $QW(A)$ of non-unipotent bi-Witt vectors exists as an extension of $QW^u(A)$.

Note that it is generally not true that $QW^u(A)$ or $QW(A)$ are rings or a $W(A)$-modules since $V$ is not a ring homomorphism. However, we see that if $A'$ is a perfect subalgebra (for instance $A'=k)$ then the Frobenius homomorphism on $W(A')$ is an isomorphism, and
\[
W(A') \otimes W_n(A) \to W_n(A); \quad x \otimes a \mapsto (F^{-n}(x)a)
\]
is a $W(A')$-module structure on $W_n(A)$ compatible with the $V$-colimits in the definition of $QW^u(A)$ and $CW^u(A)$, and \eqref{eq:unipotentwittcowittsequence} is a short exact sequence of $W(A')$-modules.

Similarily, $QW(A)$ and $CW(A)$ become $W(A')$-modules, and we have a natural short exact sequence of $W(A')$-modules
\begin{equation}\label{eq:Wittexactsequence}
0 \to W(A) \to QW(A) \to CW(A) \to 0,
\end{equation}
and $QW^u(A) = QW(A) \times_{CW(A)} CW^u(A)$.

For $A=k$, we have that $QW(k) \cong QW^u(k) \cong \Frac(W(k))$ is the field of fractions of $W(k)$, and correspondingly $CW(k) \cong CW^u(k) \cong \Frac(W(k))/W(k)$ is the injective hull of $k$ in the category of $W(k)$-modules.

The subfunctors $CW^c \subset CW$ and $QW^c \subset QW$ defined by
\[
CW^c(A) = \{ (\dots,a_{-1},a_0) \mid a_i \text{ nilpotent}  \},
\]
and similarily for $QW$, are called the groups of \emph{connected} co-Witt vectors and bi-Witt vectors, respectively. We denote by $CW^{u,c} = CW^c \cap CW^u, $ the intersection taken in $CW$. 

If $A$ is a finite $k$-algebra or, more generally, a pseudocompact $W(k)$-algebra, then the co- and bi-Witt groups $CW(A)$ and $QW(A)$ and their connected relatives carry a natural, linear topology, in which they are complete Hausdorff. An open neighborhood basis of $0$ in $CW(A)$ is given by
\[
U_{n,I} = \{ (a_i) \in CW(I) \mid a_i = 0 \text{ for } i>-n\} \quad (n \geq 0, I \triangleleft A \text{ with } A/I \text{ finite length.})
\]

We need a finitary version of Witt vectors that is covariant in the length:
\begin{defn}
Let $A$ be a ring and $\nil(A)$ its nilradical. Write $\nil^{(k)}(A)$ for $\nil(A)^{p^k}$, with $\nil^{(k)}A=\nil(A)$ for $k<0$. For $n \in \Z$, define the subgroup
\[
W^n(A) = \{(a_0,a_1,\dots) \in W(A) \mid a_i \in \nil^{(i-n)}(A)\} \subseteq W(A).
\]
One verifies easily from the definition of the Witt vectors that this is a subgroup. There are induced maps
\begin{eqnarray*}
i\colon& W^n(A) \to W^{n+1}(A) & \text{(inclusion),}\\
F\colon& W^n(A) \to W^{n-1}(A) & \text{(Frobenius, cf. \cite[Lemma 1.4]{davis-kedlaya:witt}), and}\\
V\colon& W^n(A) \to W^{n+1}(A) & \text{(Verschiebung).}
\end{eqnarray*}
Let $W^{\fin}(A)$ denote the union $\bigcup_n W^n(A).$ We call this functor from rings to abelian groups the functor of finitary Witt vectors. The reason for this name is that if $A$ is a finite ring then $W^{\fin}(A)$ consists of finitely supported Witt vectors.
\end{defn}

The functors $W^{\fin},$ $CW$, $CW^u$, $CW^{u,c},$ $CW^c$, $QW$, and $QW^c$ restrict to the category of finite $W(k)$-algebras (that is, algebras whose underlying modules are of finite length), and thus restricted are ind-representable (for the case of $CW$ and its relatives, see e.g. \cite[Ch.2, \textsection 4]{fontaine:groupes-divisibles}) and thus formal groups. 

We will sometimes for clarity use a subscript to denote which base ring they are defined over. So, for example, $CW_k$ denotes the functor $CW$ restricted to on finite $k$-algebras, while $CW_{W(k)}$ is the restriction to finite $W(k)$-algebras. Obviously, by further restriction $CW_{W(k)}$ agrees with $CW_k$ on finite $k$-algebras.

The formal schemes over $W(k)$ underlying $CW_{W(k)}$, $CW^u_{W(k)}$, $CW^c_{W(k)}, W^{\fin}_{W(k)}$ are represented by pseudo-compact $W(k)$-algebras $\Reg{CW_{W(k)}}, \Reg{CW_{W(k)}^c},\Reg{CW_{W(k)}^u}$. These rings are described in \cite[Ch.2, \textsection 3]{fontaine:groupes-divisibles}, and we recall their description here. Let $R=W(k)[x_0,x_{-1},x_{-2}, \ldots]$ be the polynomial ring in infinitely many variables and for $s \geq 0$, let $\mathfrak{J}_s = (x_{-s},x_{-s-1},\dots)$. Then the underlying pro-algebras of $\Reg{CW_{W(k)}},\Reg{CW^c_{W(k)}}$  and $\Reg{CW^u_{W(k)}}$ are given by the $W(k)$-profinite completions of the rings
\[
R/(\mathfrak{J}_r^s)_{r,s\geq0}, \quad R/(\mathfrak{J}^s_0)_{s \geq 0}, \quad \text{and} \quad R/(\mathfrak{J}_r)_{r \geq 0}.
\]
respectively. The pro-algebra of $\Reg{W_{W(k)}^{\fin}}$ is the completion of the infinite polynomial ring $W(k)[x_0,x_1, \ldots ]$ with respect to the ideals $J_0^s+J_r, r,s \geq 0,$ where $J_r = (x_r,x_{r+1},x_{r+2}, \ldots ).$

By base change to $k,$ we obtain the defining pro-algebras for $\Reg{CW_{k}},\Reg{CW^c_{k}}, \Reg{CW^u_k}, \Reg{CW^c_k}$ and $\Reg{W_k^{\fin}}.$ 

\subsection{$\Reg{W(k)}$-algebras with a lift of the Frobenius} \label{subsec:OWkalgebras}
The opposite category of topologically flat $W(k)$-modules (and -coalgebras) can be described as in \cite[Exposé VII\textsubscript{B}]{Gille:SGA3} as a subcategory of the category of $\Reg{W(k)}$-modules (resp. -algebras). We now recall these definitions.

Let $\Reg{W(k)}\colon \FAlg{W(k)} \rightarrow \FAlg{W(k)}$ be the identity functor. 
\begin{defn}
An $\Reg{W(k)}$-module ($\Reg{W(k)}$-algebra) $\underline{M}$ is a functor which to any $A \in \FAlg{W(k)}$ associates an $A$-module ($A$-algebra) $\underline{M}(A)$ and which to any morphism $\varphi\colon A \rightarrow B$ gives a morphism $$\underline{M}(\varphi)\colon B \otimes_A \underline{M}(A) \rightarrow \underline{M}(B)$$ of $B$-modules ($B$-algebras), satisfying the obvious identity and composition axioms.
We follow the regrettable choice of \cite{Gille:SGA3} to call $\underline{M}$ \emph{admissible} if $\underline M(\varphi)$ is an isomorphism for all $\varphi\colon A \rightarrow B$. We will furthermore call $\underline{M}$ \emph{flat} if it is admissible and if for any $A \in \FAlg{W(k)},$ $\underline{M}(A)$ is a flat $A$-module.
\end{defn}

Given a $\Reg{W(k)}$-algebra $\underline{C},$ we will call a morphism $\underline{a}:\Reg{W(k)} \rightarrow \underline{C}$ of $\Reg{W(k)}$-modules an \emph{element of $\underline{C}$} and will write $\underline{a} \in \underline{C}.$ When $C$ is admissible, giving an element $\underline{a} \in \underline{C}$ is the same as giving an element of $\lim_n \underline{C}(W_n(k)).$ Given an element $\underline{a} \in \underline{C}$ and a finite $W(k)$-algebra $A$, we denote by $\underline{a}|_{A} \in \underline C(A)$ the evaluation of $\underline a(A)$ at $1 \in A$.



\begin{thm}[{\cite[Proposition ~1.2.3.E.]{Gille:SGA3}}] \label{thm:antiequivflat}
The functor $\underline{I}\colon \Mod_{W(k)} \to \Mod_{\Reg{W(k)}}$ defined for $N \in \Mod_{W(k)}$ and $A \in \FAlg{W(k)}$ by
\[
\underline{I}(N)(A) = \Hom_{\Mod_{W(k)}}^c(N,A) \cong \Hom_{\Mod_{A}}^c(A \widehat{\otimes}_{W(k)} N,A)
\]
restricts to a strongly monoidal anti-equivalence between the category of topologically flat $W(k)$-modules with the completed tensor product $\hat\otimes_{W(k)}$ and that of flat $\Reg{W(k)}$-modules with the objectwise tensor product. \qed
\end{thm}

In particular, this anti-equivalence extends to one between flat $\Reg{W(k)}$-coalgebras (respectively flat $\Reg{W(k)}$-algebras) and topologically flat $W(k)$-algebras (topologically flat $W(k)$-coalgebras).

\begin{defn}
A \emph{$\mathcal O_{W(k)}$-algebras with a lift of the Frobenius} is a pair $(\underline{C},F),$ where $\underline{C}$ is a flat $\Reg{W(k)}$-algebra and $F$ is an algebra endomorphism of $\underline{C}$ such that $F(k)\colon\underline{C}(k) \rightarrow \underline{C}(k)$ is the $p$th power map on $\underline{C}(k).$ We denote by $\underline{\Alg}_{W(k)}^F$ the category consisting of these, with the obvious morphisms.
\end{defn}
Using Theorem~\ref{thm:antiequivflat}, it is easy to see that this category is dual to the category of topologically flat $W(k)$-coalgebras with a lift of the Verschiebung.

\medskip
Let $\Reg{W(k)}[\mathbf x] = \Reg{W(k)}[x_0,x_1,\dots]$ denote the flat $\Reg{W(k)}$-algebra which to any $A \in \FAlg{W(k)}$ assigns the polynomial ring $A[\mathbf x]$. For each $n \geq 0$, the element $\underline{w}_n \in \Reg{W(k)}[\mathbf x]$ is defined on $A$ by the $n$th ghost polynomial 
\[
\underline{w}_n|_A = x^{p^n}_0+px^{p^{n-1}}_1+ \cdots p^n x_n.
\]
Clearly, whenever we have a sequence of elements $\underline{\mathbf a} = (\underline{a}_0 , \underline{a}_1 , \ldots \underline{a}_n)$ of a flat $\Reg{W(k)}$-algebra $\underline{C}$, we can evaluate $\underline{w}_n$ at these elements to get an element 
\[
\underline{w}_n(\underline{\mathbf a}) \in \underline{C}.
\]

\begin{lemma}[Dwork's lemma for flat $\Reg{W(k)}$-algebras] \label{lemma:dworkoalgebras}
Let $(\underline{C},F) \in \underline{\Alg}_{W(k)}^F.$ Let $\underline{\mathbf g} = (\underline{g}_0, \underline g_1,\ldots)$ be a sequence of elements in $\underline{C}$ such that for all $n \geq 1$,
\[
{\underline{g}_{n+1}}|_{W_n(k)} = F_{W_n(k)}(\underline g_{n}|_{W_n()}) \in \underline{C}(W_n(k)).
\]
Then there are unique elements $\underline{\mathbf q}$ such that $\underline{w}_n(\underline{\mathbf q}) = \underline{g}_n$ for all $n$.
\end{lemma}
\begin{proof}
Define the algebra 
\[
\underline{C}(W(k)) = \lim_n \underline{C}(W_n(k)),
\]
where the limit is taken over the canonical projection maps 
\[
p_n\colon\underline{C}(W_n(k)) \rightarrow \underline{C}(W_{n-1}(k)).
\]
Since $\underline{C}$ is admissible, we have that 
\[
\underline{C}(W_{n-1}(k)) \cong \underline{C}(W_n(k)) \otimes_{W_n(k)} W_{n-1}(k)
\]
so that the maps $p_n$ are surjective. Furthermore, since each $\underline{C}(W_n(k))$ is flat as a $W_n(k)$-module, $\underline{C}(W(k))$ is torsion free over $W(k).$ The endomorphism $F$ yields an endomorphism $\underline{C}(W(k)) \rightarrow \underline{C}(W(k)),$ and the elements $\underline{g}_n$ yield honest elements $g_n \in \underline{C}(W(k)).$ We now claim that 
\[
\underline{C}(W(k)) \otimes_{W(k)} W_n(k)\cong \underline{C}(W_n(k)).
\]
 Indeed, let 
\[
0 \rightarrow \underline{C}(W(k)) \rightarrow \prod_{i=1}^{\infty} \underline{C}(W_i(k)) \rightarrow \prod_{i=1}^\infty \underline{C}(W_i(k)) \rightarrow 0
\]
 be the exact sequence defining the limit, where $\lim^1$ vanishes because the system is Mittag-Leffler. Note that since $W_n(k)$ is a finitely presented $W(k)$-module, tensoring with $W_n(k)$ commutes with infinite products. Since 
\[
\underline{C}(W_i(k)) \otimes_{W(k)} W_n(k) \cong \underline{C}(W_n(k))
\]
 for $i \geq n,$ the limit of the system $\{\underline{C}(W_i(k)) \otimes_{W(k)} W_n(k)\}_i$ is $\underline{C}(W_n(k))$. Thus to show that $\underline{C}(W(k)) \otimes_{W(k)} W_n(k)$ coincides with this limit, it is enough to show that the induced map 
\[
\Tor^1_{W(k)}\Bigl(\prod_{i=1}^\infty \underline{C}(W_i(k)),W_n(k)\Bigr) \rightarrow \Tor^1_{W(k)} \Bigl(\prod_{i=1}^\infty \underline{C}(W_i(k)),W_n(k)\Bigr)
\]
 is surjective. But this follows since the inverse system $\{\underline{C}(W_i(k))[p^n]\}_i$ is Mittag-Leffler. Indeed, by flatness we have that 
\[
\underline{C}(W_i(k))[p^n] = p^{i-n} \underline{C}(W_i(k)),
\]
 where it is understood that $p^{i-n}=1$ if $i \leq n.$ But given $j,$ this implies that the map 
\[
\underline{C}(W_{j+n}(k))[p^n] = p^j \underline{C}(W_{j+n}(k)) \to p^{j-n} \underline{C}(W_j(k)) = \underline{C}(W_j(k))[p^n]
\]
 is zero, since $p^j=0$ in the latter algebra. In conclusion, $\underline{C}(W(k))$ is a torsion-free $W(k)$-algebra  with a lift of the Frobenius. The usual Dwork lemma, applied to $\underline{C}(W(k))$, gives a sequence $\mathbf q \in \underline{C}(W(k))$ such that $w_n(\mathbf q) = g_n.$ This sequence gives elements $\underline{\mathbf q}$ satisfying the requirements of the lemma. Conversely, we see that any sequence of elements $\underline{\mathbf q}$ arises in this manner. 
\end{proof}

Let $\WtH,$ be the $\Reg{W(k)}$-Hopf algebra which to any $A \in \FAlg{W(k)}$ associates the Hopf algebra $A[x_0, x_1, x_2,\ldots]$ representing the functor taking a $A$-algebra to its ring of $p$-typical Witt vectors. It comes with a Frobenius lift $F_{\WtH}.$ We denote by $\underline{\operatorname{Hopf}}_{W(k)}^F$ the category of pairs $(\underline{H},F_{\underline H}),$ where $\underline{H}$ is a $\Reg{W(k)}$-Hopf algebra and $F_{\underline H}$ is a lift of the Frobenius, with the obvious morphisms. Given a $\Reg{W(k)}$-Hopf algebra $\underline{H},$ we say that an element $\underline{x} \in \underline{H}$ is \emph{primitive} if $\underline{x}|_A$ is a primitive element of $\underline{H}(A)$ for each $A \in \FAlg{W(k)}.$ We denote by $P \underline{H}$ the primitives of a $\Reg{W(k)}$-Hopf algebra.

\begin{lemma} \label{lemma:mapsfromWitt}
Let $\underline{H} \in \underline{\operatorname{Hopf}}_{W(k)}^F$. Then there is a natural isomorphism 
\[
e\colon \Hom_{\underline{\operatorname{Hopf}}_{W(k)}^F}(\WtH,\underline{H}) \xrightarrow{\cong} P \underline{H}
\]
 given by $e(f)=f(\underline{x_0}).$ 
\end{lemma}
\begin{proof}
This proof is the same as the proof of \cite[Proposition ~1.9]{goerss:hopf-rings}; we include a sketch for the reader's convenience. The map $e$ is injective by Lemma \ref{lemma:dworkoalgebras} and the fact that $\underline{F}^n_H(f(\underline{x_0})) = f(\underline{w}_n).$ For the surjectivity, one notes that given a primitive $y \in P \underline{H},$ all the elements $F_{\underline H}^n(y)$ are primitive as well. Using Dwork's lemma again, we get a map $f\colon \WtH \rightarrow \underline{H}$ such that $f(\underline{w}_n) = F_{\underline H}^n(y).$ It remains to show is that $f$ is a morphism of $\Reg{W(k)}$-Hopf algebras and that $f$ commutes with the Frobenius. Both of these are direct applications of the uniqueness of Lemma \ref{lemma:dworkoalgebras}.
\end{proof}

\subsection{Proof of Theorem~\ref{thm:indecomposableequiv}}

There are two objects in $\FgpsV$ which will play a key role in the proof of \ref{thm:indecomposableequiv}, introduced to the reader in \ref{subsec:backgroundWitt}. The first is the functor 
\[
CW^c_{W(k)}\colon \FAlg{W(k)} \rightarrow \Ab,
\]
 the functor of connected co-Witt vectors, with the Verschiebung lift $V_{CW^c_{W(k)}},$ where 
 \[
 V_{CW^c_{W(k)}}(\dots,a_{-2},a_{-1},a_0) = (\dots,a_{-2},a_{-1}).
 \]
The other object is the functor $W^{\fin}_{W(k)}$ with the Verschiebung lift $V_{W^{\fin}_{W(k)}}.$ For $A \in \FAlg{W(k)},$ the Verschiebung acts on an element $(a_0,a_1,a_2, \ldots)$ of $W^{\fin}_{W(k)}(A)$ by shifting to the right. 
We have that 
\[
Q(W^{\fin}_{W(k)}) \cong \prod_{i=0}^\infty W(k)
\]
where
\[
V(x_0, x_1,x_2, \ldots ) = (F^{-1}_{W(k)}(x_1), F^{-1}_{W(k)}(x_2) , \dots) \in \prod_{i=0}^\infty W(k).
\]
On the other hand, $Q(CW_{W(k)}^c) \cong \widehat{\bigoplus}_{i=0}^\infty W(k)$, where $\widehat{\bigoplus}$ denotes the profinite completion of the direct sum. In this situation, $V$ acts by shifting to the right and taking $F^{-1}_{W(k)}$ of the components.

\begin{lemma}\label{lemma:homsofindecomposables}
For $M \in \mathcal{M}_V$, there are natural isomorphisms 
\[
\Hom_{\mathcal{M}_V}(M,Q(W^{\fin}_{W(k)})) \cong \Hom_{W(k)}^c(M,W(k)) \cong M^*
\]
and
\[
\Hom_{\mathcal{M}_V}(Q(CW^c_{W(k)}),M) \cong \Hom^c_{W(k)}(W(k),M) \cong M.
\]
\end{lemma} 
\begin{proof}
Let $f\colon M \rightarrow Q(W^{\fin}_{W(k)}) \cong \prod_{i=0}^\infty W(k)$ be a map in $\mathcal M_V.$ Writing 
\[
f(a) = (f_0(a),f_1(a),\ldots),
\]
 we see that 
\[
f(V^i a ) = (F^{-i}_{W(k)}(f_i(a)),F^{-i}_{W(k)}( f_{i+1}(a)), \ldots),
\]
 that is, $f_0 \circ V^i = F^{-i}_{W(k)} \circ f_i$. Thus $f \mapsto f_0$ gives the first natural isomorphism.

For the second claim, let $f\colon Q(CW^c_{W(k)}) \rightarrow M$ be a map in $\mathcal M_V$.  Since $Q(CW^c_{W(k)}) \cong \widehat{\bigoplus}_{i=0}^\infty W(k)$ is free in the sense that any (not necessarily continuous) homomorphism $\bigoplus_{i=0}^\infty W(k) \rightarrow M $ extends to a continuous homomorphism from $Q(CW^c_{W(k)}),$ giving a $V$-linear homomorphism $f\colon \bigoplus_{i=0}^\infty W(k) \rightarrow M$ is the same as giving a map $W(k) \rightarrow M$.
\end{proof}

\begin{lemma} \label{lemma:indecomposablesfaithful}
For $G,H \in \FgpsV,$ the natural map 
\[
\Hom_{\FgpsV}(G,H) \rightarrow \Hom_{\mathcal{M}_V}(Q(H),Q(G))
\]
 is an injection. 
\end{lemma}
\begin{proof}
Arguing as in {\cite[Proposition ~2.10]{goerss:hopf-rings}}, we denote by $\FGpsc{QW(k)}$ the category of connected formal groups over the fraction field $QW(k)$ of $W(k)$ and by $\mathcal{M}_{QW(k)}$ the category of topologically free vector spaces over $QW(k)$.  For a formal group $G$ over $W(k),$ we let $G_{QW(k)}$ be the base change to $QW(k)$ and will use the same notation for the base change of an element of $\mathcal{M}_V$ to $\mathcal{M}_{QW(k)}.$  We then have a diagram
\[
\begin{tikzcd}[column sep = small]  \Hom_{\FgpsV} (G,H) \arrow[d] \arrow[r] & \Hom_{\mathcal{M}_V} (Q(H),Q(G)) \arrow[d] \\ \Hom_{\FGpsc{QW(k)}}(G_{QW(k)}, H_{QW(k)}) \arrow[r] & \Hom_{\mathcal{M}_{QW(k)}} (Q(H)_{QW(k)}, Q(G)_{QW(k)}). \end{tikzcd}
\]
 The left hand vertical map is an injection since $G$ and $H$ are smooth, and so is the right hand vertical map since $\Reg{G}$ and $\Reg{H}$ are power series rings. The lower horizontal map is a bijection by \cite[Chapitre II, Proposition 10.6]{fontaine:groupes-divisibles}, thus the upper horizontal map is injective. 
\end{proof}

\begin{lemma} \label{lemma:wittindecomposables}
If $H \in \FgpsV,$ then 
\[
Q\colon\Hom_{\FgpsV}(W^{\fin}_{W(k)},G) \rightarrow \Hom_{\mathcal{M}_V}(QG,Q(W^{\fin}_{W(k)})) \cong Q(H)^* 
\]
 is a bijection.
\end{lemma}
\begin{proof}
By Lemma~\ref{lemma:indecomposablesfaithful}, this map is an injection. To prove surjectivity, we first note that the functor $\underline I$ of Thm.~\ref{thm:antiequivflat} extends to an anti-equivalence between the category $\FgpsV$ and the category of $\Reg{W(k)}$-Hopf algebras with a lift of the Frobenius \cite[Exposé VII\textsubscript{B}, 2.2.1]{Gille:SGA3}. Under this duality, $W^{\fin}_{W(k)}$ with its Verschiebung lift corresponds to the $\Reg{W(k)}$-Hopf algebra $\WtH$ of Lemma~\ref{lemma:mapsfromWitt}. To see this, note that for any finite $W(k)$-algebra $A,$ the Artin-Hasse exponential gives a map $$\WtH(A) \rightarrow \underline{I}(W^{\fin}_{W(k)})(A).$$ By \cite[V, \textsection 4, n.\kern-0.5bp\textsuperscript{o} 4, Cor. 4.6]{demazure-gabriel:groupes-algebriques-1}, this map is an isomorphism modulo $p.$ Thus, the original map $\WtH(A) \rightarrow \underline{I}(W^{\fin}_{W(k)})(A)$ is an isomorphism by two applications of Nakayama's lemma, the first application showing that it is surjective, and then using the projectivity of $\underline{I}(W^{\fin}_{W(k)})(A)$ to show that it is injective. By duality, an element
\[
x \in Q(G)^* = \Hom^c_{W(k)}(Q(G),W(k))
\]
gives a coherent family of primitives in $\underline{I}(G).$ By Lemma~\ref{lemma:mapsfromWitt}, this coherent family of primitives gives a map $\underline{I}(G) \rightarrow \WtH  \cong \underline{I}(W^{\fin}_{W(k)})$, and upon dualization we get a map $f:W^{\fin}_{W(k)} \rightarrow H$ of formal groups such that $Qf = x.$ 
\end{proof}

In preparation of the next proposition, note that the forgetful functor $\mathcal{M}_V \rightarrow \mathcal{M}$ to the category of topologically free modules over $W(k)$ has a right adjoint $J,$ taking $M \in \mathcal{M}$ to $JM = \prod_{i=0}^\infty M$ with $V$ acting by
\[
V(m_0, m_1, m_2, \ldots ) = (F^{-1}_{W(k)}(m_1),F^{-1}_{W(k)}(m_2), \ldots),
\]
where by abuse of notation, we write $F_{W(k)}$ for the automorphism of $M$ that corresponds to the Frobenius on each component under the isomorphism $M \cong \prod_i W(k)$. We will denote by $\mathbf{PC}_{W(k)}$ the category of pseudocompact modules over $W(k).$ We have a forgetful functor $\mathcal{M}_V \rightarrow \mathbf{PC}_{W(k)}.$ 

\begin{prop} \label{prop:indecomposableadjoint}
The  functor $Q\colon \left(\FgpsV\right)^\op \rightarrow \mathcal{M}_V$ has a left adjoint $\FrM$. For $M \in \mathcal{M}_V$, the counit map $Q \FrM(M) \rightarrow M$  is an isomorphism.
\end{prop}
\begin{proof}
We start by defining the functor $\FrM$. Given $M \in \mathcal{M}_V$, we have a functorial resolution
\begin{equation}\label{eq:funcresolution}
0 \rightarrow M \xrightarrow{\eta} JM \xrightarrow{d} JM \rightarrow 0
\end{equation}
where
\begin{align*}
\eta(m) &= (m,F_{W(k)}(Vm),F^2_{W(k)}(V^2m), \ldots),\\
d(m_0,m_1,m_2, \ldots) & = (m_1-F_{W(k)}(Vm_0),m_2-F_{W(k)}(Vm_1), \ldots).
\end{align*}
Choosing an isomorphism of $M \cong \prod_\alpha W(k)$, we get an isomorphism $JM \cong \prod_\alpha Q(W^{\fin}_{W(k)})$. By Lemma \ref{lemma:wittindecomposables}, there is a map
\begin{equation}\label{eq:fmap}
f\colon \bigoplus_\alpha W^{\fin}_{W(k)} \rightarrow \bigoplus_\alpha W^{\fin}_{W(k)}
\end{equation}
such that $Qf = d.$ We define $\FrM(M)$ to be the pushout 
\[
\begin{tikzcd} \bigoplus_\alpha W^{\fin}_{W(k)}  \arrow[r,"f"] \arrow[d] & \bigoplus_\alpha W^{\fin}_{W(k)} \arrow[d,"g"] \\ \Spf{W(k)} \arrow[r] & \FrM(M) \end{tikzcd}
\]
in the category $\FgpsV$. An argument is required to see that this pushout exists. The resolution \eqref{eq:funcresolution} stays exact after applying the forgetful functor $\mathcal{M}_V \rightarrow \mathbf{PC}_{W(k)}.$  This implies that the map \eqref{eq:fmap} is injective in the category of (not necessarily topologically flat) formal groups over $W(k),$ since the kernel $\ker f$ is a connected group satisfying $Q(\ker f) = 0.$
By \cite[2.4, Théorème]{Gille:SGA3} the pushout $\FrM(M)$ exists, is topologically flat, and the map $g$ is faithfully flat, which implies that $\FrM(M)$ is smooth. Extending this construction in the obvious way to morphisms between objects in $\mathcal{M}_V,$ we thus get a functor $\FrM:\mathcal{M}_V \rightarrow \FgpsV.$ 

Let us now show that $\FrM$ is adjoint to $Q$ and that the counit $M \rightarrow Q \FrM(M)$ is an isomorphism. We start by showing that $Q(g)\colon Q\FrM(M) \rightarrow Q(\bigoplus_\alpha W^{\fin}_{W(k)})$ is an injection. Consider the base change of the exact sequence $$0 \rightarrow \bigoplus_\alpha W^{\fin}_{W(k)} \xrightarrow{f} \bigoplus_\alpha W^{\fin}_{W(k)} \rightarrow \FrM(M) \rightarrow 0$$ in the category of formal groups over $W(k)$ to $k.$  We obtain an exact sequence of formal groups over $k,$ $$0 \rightarrow \bigoplus_\alpha W^{\fin}_k \xrightarrow{f_k} \bigoplus_\alpha W^{\fin}_k \rightarrow \FrM(M)_k \rightarrow 0.$$ 
Since all the groups involved are smooth, $Q$ is exact, and we get that
\[
Q(\FrM(M)_k) \cong Q(\FrM(M)) \hat{\otimes} k \rightarrow Q(\bigoplus_\alpha W^{\fin}_k) \cong Q(\bigoplus_\alpha W^{\fin}_{W(k)}) \hat{\otimes} k
\]
is injective. Since $Q(\bigoplus_\alpha W^{\fin}_{W(k)})$ is topologically flat, the map $g\colon Q(\FrM(M)) \rightarrow Q(\bigoplus_\alpha W^{\fin}_{W(k)})$ is such that $\ker g \hat{\otimes} k \rightarrow Q(\FrM(M)) \hat{\otimes} k$ is injective and hence $\ker(g) \hat{\otimes} k = 0$.

 By Nakayama's lemma  we then have that $\ker g = 0.$ Note that this implies that $Q(\Fr(M)) \cong M,$ since both are kernels of the map $Qf=d.$  For an object $G \in \FgpsV$ we get a diagram $$\begin{tikzcd} 0 \arrow[d]  & 0 \arrow[d] \\ \Hom_{\FgpsV} (\FrM(M),G) \arrow[d]   \arrow[r] & \Hom_{\mathcal{M}_V}(Q(G),Q(\FrM(M))) \arrow[d] \\ \Hom_{\FgpsV}(\bigoplus_\alpha W^{\fin}_{W(k)}, G) \arrow[d]  \arrow[r]  & \Hom_{\mathcal{M}_V} (Q(G),Q(\bigoplus_\alpha W^{\fin}_{W(k)})) \arrow[d] \\ \Hom_{\FgpsV} (\bigoplus_\alpha W^{\fin}_{W(k)},G)   \arrow[r]  & \Hom_{\mathcal{M}_V} (Q(G),Q(\bigoplus_\alpha W^{\fin}_{W(k)})) \end{tikzcd}$$ 
where the middle two horizontal maps are isomorphisms. This yields that $$\Hom_{\FgpsV} (\FrM(M),G) \cong \Hom_{\mathcal{M}_V}(Q(G),Q(\FrM(M))).$$ But as previously noted, $ M \cong Q \FrM(M),$ so we have our adjunction. 
\end{proof}
\begin{proof}[Proof of Thm.~\ref{thm:indecomposableequiv}]
From Prop.~\ref{prop:indecomposableadjoint}, the counit map $Q \FrM(M) \rightarrow M$ is an isomorphism. To show that the unit map 
\[
G \rightarrow \FrM(Q(G))
\]
 is an isomorphism, we note that the map 
\[
Q \FrM(Q(G)) \rightarrow Q(G)
\]
 is an isomorphism. But this implies that $G \rightarrow \FrM(Q(G))$ is an isomorphism, since both groups are smooth and connected. 
\end{proof}

\section{Dieudonné theory} \label{sec:dieudonne}
Let $k$ be a perfect field of characteristic $p>0$. In this section, we will recall the Dieudonné theory needed for our purposes (cf. \cite{fontaine:groupes-divisibles,demazure-gabriel:groupes-algebriques-1,demazure:pdivisiblegroups}). Let $\phi\colon W(k) \to W(k)$ denote the Witt vector Frobenius map, which is a lift of the $p$th power map on $k$ to $W(k)$. By the perfectness of $k$, this is an isomorphism. A similar isomorphism exists on $CW(k)$ and all other related versions of Witt vectors.

\begin{defn}
Let $\mathcal R$ be the noncommutative ring generated over $W(k)$ by two indeterminates $F$, $V$ modulo the relations
\[
FV=VF=p, \; Fa = \phi(a) F,\; aV = V\phi(a) \quad \text{for $a \in W(k)$}.
\]
The category of \emph{Dieudonn\'e modules} $\DMod{k}$ is the category of left $\mathcal R$-modules. 
Denote by $\mathcal F$ (resp. $\mathcal V$) the subrings of $\mathcal R$ generated by $W(k)$ and $F$ (resp. $W(k)$ and $V$)
\end{defn}

Given a Dieudonné module $M$ we define the dual $I(M)$ to be $\Hom_{W(k)}(M,CW(k))$. This becomes a Dieudonné module by
\[
F(\alpha)(m) = \phi (\alpha(Vm)) \quad \text{and} \quad V(\alpha)(m) = \phi^{-1}(\alpha(Fm))
\]
for $\alpha \in \Hom_{W(k)}(M,CW(k))$ and $m \in M$.

For a module $M \in \DMod{k},$ let $M[V^n]$ be the kernel of multiplication by $V^n.$ Let $ \DModV{k}$ be the full subcategory of $\DMod{k}$ consisting of those $M$ such that $M \cong \colim_n M[V^n].$

\begin{thm}\label{thm:dieudonneforaffinegroups}\cite[V.1.4, Théorème 4.3]{demazure-gabriel:groupes-algebriques-1}
There is an anti-equivalence of categories
\[
D\colon \{\text{unipotent groups over $k$}\} \to  \DModV{k}
\] given by taking the unipotent group $G$ to \[  D(G)= \colim_n \Hom_{\AbSch{k}}(G, W^n_k). \]
Furthermore, $D(G)$ is finitely generated if and only if $G$ is of finite type, and $D(G)$ is of finite length as a $W(k)$-module if and only if $G$ is a finite affine group. \qed
\end{thm}
Since the category of affine groups over a perfect field splits into the product of the full subcategories of unipotent groups and of multiplicative groups, Theorems~\ref{thm:multiplicativegroupclassification} and \ref{thm:dieudonneforaffinegroups} give an anti-equivalence between the category of affine groups and the product of the categories $\DModV{k}$ and $\Mod_\Gamma$.

Let us call an affine group \emph{$p$-adic} if $G$ is the inverse limit of the affine groups $\coker( G \xrightarrow{p^n} G).$ We note that any unipotent group is $p$-adic since $VF=p.$ Let $\DMod{k}^p$ denote the full subcategory of \emph{locally $V$-finite} modules in $\Mod_{\mathcal R}$, that is, those $M$ such that every $x \in M$ is contained in a $\mathcal V$-submodule of finite length.

\begin{thm}\label{thm:dieudonneforpadic} 
There is an anti-equivalence of categories
\[
D\colon \{\text{$p$-adic affine groups over $k$}\} \to \DMod{k}^p
\]
given by 
\[
G \rightarrow  D(G) = \colim_n \Hom_{\AbSch{k}}(G, W^n_k) \times I(\colim_{n'} \Hom_{\AbSch{k}}((G^m)^*,W^{n'}_k))
\]
where $(G^m)^\ast$ is the Cartier dual of the multiplicative part of $G,$ and $I$ is the dual Dieudonné module.  
\end{thm}
\begin{proof}
This theorem is essentially proved in \cite[Chapter III, 6]{demazure:pdivisiblegroups}, but only for unipotent groups and for \emph{finite} multiplicative groups of $p$-power order. We claim that a $p$-adic affine group $G$ of multiplicative type is an inverse limit of finite multiplicative groups of $p$-power order, from which the theorem in the stated form follows. Indeed, since $G$ is the inverse limit of finitely generated multiplicative groups, we may  assume that $G$ is finitely generated multipliative. We also assume that $k$ is algebraically closed and thus that $G \cong k[N]$ for some finitely generated abelian group $N$ with $p^nN=0$ for some $n \geq 0$. Thus $N$ is finite and $G$ is a finite affine group.
\end{proof}
Thm.~\ref{thm:dieudonneforpadic} is also proved in \cite{hopkins-lurie:ambidexterity}, using slightly different methods from the ones in \cite{demazure:pdivisiblegroups}.

\medskip
We will now briefly discuss the dual picture of formal groups. A module $M \in  \Mod_{\mathcal R}$ is called \emph{$\mathcal F$-profinite} if $M$ is profinite as a $W(k)$-module and has a fundamental system of neighborhoods consisting of $\mathcal F$-modules. We call $M$ \emph{connected} if $F$ is topologically nilpotent on $M$, and \emph{étale} if $F$ is bijective on $M$. We denote the category of $\mathcal F$-profinite $\mathcal{R}$-modules by $\DDModF{k}$ and the full subcategories of connected (resp. étale) $\mathcal{R}$-modules by $\DDModFc{k}$ (resp. $\DDModFet{k}$).

Essentially by a Fitting decomposition, any $\mathcal F$-profinite module $M$ splits as a direct sum $M^{\operatorname{et}} \oplus M^c$ where the first summand is étale and the second is connected. Indeed, any $\mathcal F$-module of finite length splits as a direct sum of two modules where $F$ acts bijectively on one component and nilpotently on the other. Taking the limit, we get the claimed decomposition. 

A \emph{formal $p$-group} is a commutative formal group $G$ such that the canonical map $\colim_n G[p^n] \rightarrow G$ is an isomorphism, where $G[p^n]$ denotes the kernel of multiplication by $p^n.$ Note that any connected formal group is a formal $p$-group since $FV = VF = p$. We let $\FGps{k}^p$ be the category of formal $p$-groups over $k.$ 

A dual version of Thm.~\ref{thm:dieudonneforpadic} is proved in \cite[Chapter III]{fontaine:groupes-divisibles}:
\begin{thm}\label{thm:dieudonneforformalgroups}
There is an equivalence of categories
\[
\DieuFor: \FGps{k}^p \to \DDModF{k}
\] 
where $$\DieuFor(G) = \Hom_{\FGps{k}}(G,CW_k).$$ 

Under this equivalence, the full subcategories of connected formal groups and étale formal $p$-groups corresponds to the full subcategories of connected and étale $\mathcal R \text{-modules}$, respectively.  \qed
\end{thm} 
We will show later (Thm.~\ref{thm:CartierMatlis}) in which sense the equivalences of Thms.~\ref{thm:dieudonneforpadic} and \ref{thm:dieudonneforformalgroups} are dual to each other.

\section{The Dieudonné module of a smooth connected formal group with a lift of the Verschiebung}

The goal of this section is to describe the Dieudonné module of those formal groups over $k$ that are mod-$p$ reductions of smooth connected formal groups over $W(k)$ with a Verschiebung lift. These constructions are very similar to the ones of \cite{goerss:hopf-rings}.
For $G \in \FgpsV,$ we define a map $$Q(G) \rightarrow \DieuFor(G_k)$$ as follows. By Thm.~\ref{thm:indecomposableequiv} and Lemma \ref{lemma:wittindecomposables} we have an isomorphism
\[
Q(G) \cong \Hom_{\FgpsV}(G,CW^c_{W(k)}).
\]
Composing this isomorphism with the mod-$p$ reduction
\[
\Hom_{\FgpsV}(G,CW^c_{W(k)}) \rightarrow \Hom_{\FGps{k}}(G_k,CW^c_k) =\DieuFor(G_k),
\]
we get a $\mathcal V$-linear map $\eta\colon Q(G) \rightarrow \DieuFor(G_k)$ in $\mathcal{M}_V.$ By extension of scalars, we get a homomorphism $$\epsilon_G\colon \mathcal{R} \otimes_{\mathcal V} Q(G) \to \DieuFor(G_k).$$ Let $$\alpha\colon \mathcal{F} \otimes_{W(k)} Q(G) \rightarrow \mathcal{R} \otimes_{\mathcal{V}} Q(G)$$ be the natural map. We let $\mathcal{R} \hat{\otimes}^F_{\mathcal V} Q(G)$ be the completion of $\mathcal{R} \otimes_{\mathcal V} Q(G)$ with respect to the image under $\alpha$ of the $\mathcal F$-submodules $$\langle F^n\rangle  \otimes_{W(k)} Q(G)  + \mathcal F \otimes_{W(k)} M \subset \mathcal F \otimes_{W(k)} Q(G)$$ where $M \subset Q(G)$ ranges over a basis of neighborhoods of the topologically free $W(k)$-module $Q(G).$ 
Note that $\epsilon_G$ is continuous and that $\mathcal{R} \hat{\otimes}^F_{\mathcal V} Q(G)$ is $\mathcal F$-profinite as a Dieudonné module.  
\begin{thm} \label{thm:indecomposablesdieudonne}
Let $G \in \FgpsV.$ Then the map  
\[
\epsilon_G: \mathcal{R} \hat{\otimes}^F_{\mathcal V} Q(G) \rightarrow \DieuFor(G_k)
\]
 is an isomorphism.  
\end{thm}
\begin{proof}[Proof of Thm.~\ref{thm:indecomposablesdieudonne}]
Denote the source of the map $\epsilon_G$ by $E(G).$ We note that both $E(G) = \varprojlim_n E(G) /F^nE(G)$ and $\DieuFor(G_k) = \varprojlim_n \DieuFor(G_k)/F^n \DieuFor(G_k).$ It is clear that the map $\epsilon_G$ becomes an isomorphism modulo $F.$ We claim that this implies that $\epsilon_G$ is an isomorphism. That $\coker \epsilon_G = 0$ follows from observing that $$\coker \epsilon_G = F^i  \coker \epsilon_G, \qquad i \geq 0$$ so that $$\coker \epsilon_G = \varprojlim_n \coker \epsilon_G /F^n \coker \epsilon_G =0,$$ as claimed. We thus have an exact sequence $$0 \rightarrow \ker \epsilon_G \rightarrow E(G) \rightarrow \DieuFor(G_k) \rightarrow 0,$$ and, noting that $F$ acts injectively on $\DieuFor(G_k),$  this sequence is exact after reducing mod $F,$ since the derived functor of taking the quotient modulo $F$ is the functor taking a module $M$ to the elements of $M$ annihilated by $F.$ This implies that $\ker \epsilon_G=0$ so that $\epsilon_G$ is an isomorphism. 
\end{proof}
\begin{example} \label{ex:dieudonneW}
Let $W^{\fin}_k$ be the connected formal group over $k$ taking a finite $k$-algebra $R$ to the finitely supported Witt vectors. Denote by $F_{W(k)}$ the canonical Frobenius lift on the ring of Witt vectors. Using the above theorem, 
\[
\DieuFor(W^{\fin}_k) \cong \prod_{i=0}^\infty W_i(k).
\]
The Dieudonné module structure on the right hand side is given by
\begin{align*}
V(a_0,a_1,\dots) &= (F^{-1}(a_1),  F^{-1}(a_2), \ldots)\\
\intertext{and}
F(a_0,a_1,\dots) &= (0, pF(a_0),pF(a_1), \ldots).
\end{align*}
Here, $F^{-1}\colon W_i(k) \to W_{i-1}(k)$ denotes the inverse of the Frobenius on $W_i(k)$ followed by projection to $W_{i-1}(k)$ and $pF(a_i)\colon W_i(k) \to W_{i+1}(k)$ denotes the Frobenius on $W_i(k)$ followed by multiplication by $p$ from $W_i(k)$ to $W_{i+1}(k)$.

This follows from the isomorphism of $\mathcal V$-modules
\[
Q(W^{\fin}_{W(k)}) \cong \prod_{i=0}^\infty W(k),
\]
by inducing up.
\end{example} 

\section{Cartier and Matlis duality} \label{sec:duality}

As studied in Section~\ref{sec:dieudonne}, Dieudonné theory gives us two functors
\begin{align*}
\DieuFor\colon \{ \text{formal $p$-groups} \} &\rightarrow \DDMod{k}^F
\shortintertext{and}
D\colon \{\text{$p$-adic affine groups over $k$}\} &\rightarrow  \DMod{k}^p.
\end{align*}
Since the category of formal $p$-groups is dual to the category of $p$-adic affine groups, one might ask how these two functors are related. For finite affine groups, this is treated in \cite{fontaine:groupes-divisibles}. To state more precisely what we will investigate in this section, let $M \in \DDMod{k}^F.$ Then one can endow $$\Hom^c_{W(k)}(M,CW(k))$$ with an action of $\mathcal{R}$ such that $\Hom^c_{W(k)}(M,CW(k)) $ is an object of $\DMod{k}^p.$ This yields a contravariant functor $$I^c: \DDMod{k}^F \rightarrow \DMod{k}^p.$$  Conversely, if $N \in  \DMod{k}^p,$ then there is a natural topology on $\Hom_{W(k)}(N,CW(k)),$ and there is an action of $\mathcal{R}$ on it that makes $\Hom_{W(k)}(N,CW(k))$ into an object of $\DDMod{k}^F.$ We obtain in this manner a functor $$I: \DMod{k}^p \rightarrow \DDMod{k}^F,$$ and using Matlis duality (Thm.~\ref{thm:matlisduality}), one shows that the functors $I,I^c$ give a duality between $\DDMod{k}^p$ and $\DMod{k}^F.$ On the other hand, taking a p-adic affine group $G$ to its Cartier dual $G^*$ we get a formal $p$-group, and conversely, if $H$ is a formal $p$-group, the Cartier dual $H^*$ is a p-adic affine group.  We will denote the functor taking a formal $p$-group to its Cartier dual by $(-)^*,$ and the same notation will be used for the functor taking a $p$-adic affine group to its Cartier dual. It is well-known (see for example  \cite[I,\textsection 5]{fontaine:groupes-divisibles}) that Cartier duality gives a duality between the category of formal $p$-groups and $p$-adic affine groups. A natural question to ask is how Cartier duality relates to the duality between $\DMod{k}^p$ and $\DDMod{k}^F$ we just gave. The answer is given by the following theorem, whose proof will occupy the rest of this section.
\begin{thm} \label{thm:CartierMatlis}
The following diagram is $2$-commutative $$\begin{tikzcd} \{\text{$p$-adic affine groups over } k\} \arrow[r,"D"] \arrow[d,"(-)^*"] & \DMod{k}^p \arrow[d,"I"] \\  \{ \text{formal $p$-groups over } k \} \ \arrow[r,"\DieuFor"] & \DDMod{k}^F .\end{tikzcd}$$ Equivalently, Cartier duality commutes with Matlis duality, in the sense that if $G$ is a $p$-adic affine group, then $ID(G)$ is naturally isomorphic to $\DieuFor(G^*).$ Conversely, if $G$ is a formal $p$-group, then $I^c(\DieuFor(G))$ is naturally isomorphic to $D(G^*).$    
\end{thm}
\subsection{Matlis duality and monoidal structures on torsion modules}
This subsection recalls the classical theory of Matlis duality and we introduce a monoidal structure on the category of torsion $W(k)$-modules which will help us describe the Dieudonné module of the tensor product of two $p$-adic affine groups. Recall that $CW$ denotes the functor taking a algebra to its co-Witt vectors.
\begin{prop}\label{prop:cowittmultiplication-perfectalg}
For any \emph{perfect} finite $k$-algebra $A$, there are natural morphisms
\[
CW(k) \to CW(A) \quad \text{and} \quad \Tor_{W(k)}(CW(A),CW(A)) \to CW(A)
\]
making $CW(A)$ a monoid in the category of artinian $W(k)$-modules with monoidal structure $\Tor_{W(k)}$ abd unit $CW(k))$.
\end{prop}
\begin{proof}
For ease of notation, we will denote $\Tor_{W(k)}(M,N)$ by $M * N$. Tensoring the short exact sequence \eqref{eq:Wittexactsequence} with $CW(A)$, we obtain a connecting homomorphism
\[
CW(A) * CW(A) \to W(A) \otimes CW(A),
\]
which composes with the $W(A)$-module structure on $CW(A)$ to the desired multiplication.
\end{proof}

\begin{thm}[Matlis duality] \label{thm:matlisduality}
The functor $I = \Hom_{W(k)}(-,CW(k))$ is an anti-equivalence between the category of torsion $W(k)$-modules and the category of pseudocompact $W(k)$-modules.
\end{thm}
\begin{proof}
Classical Matlis duality establishes this functor as a duality between Noetherian and Artinian modules. The ring $W(k)$ is a discrete valuation ring, hence a PID, and hence a module over it is Noetherian if and only if it is finitely generated, and modules which are both Noetherian and Artinian are precisely those of finite length. Since every torsion module is the colimit of its finite length submodules, the result follows by passing to ind- resp. pro-categories.
\end{proof}

The inverse functor $I^{c}$ from pseudocompact $W(k)$-modules to torsion $W(k)$-modules is given by
\[
I^c(K) = \Hom^c_{W(k)}(K,CW(k)) = \colim_j \Hom_{W(k)}(K_j,CW(k))
\]
if $K$ is represented by an inverse system $K\colon J \to \Mod_{W(k)}$.

In our setting, Matlis duality is monoidal with respect to the \emph{derived} tensor product. More explicitly, we have the following two symmetric monoidal structures:
\begin{itemize}
\item On the category of pseudocompact $W(k)$-modules: the (completed) tensor product $\otimes$ with unit the pseudocompact module $W(k)$. 
\item On the category of torsion $W(k)$-modules: the torsion product $M * N = \Tor^1_{W(k)}(M,N)$  with unit $CW(k)$. Indeed, note that if $M$ is torsion then $M \otimes QW(k)=0 = M * QW(k)$ and hence the connecting homomorphism $M * CW(k) \to M \otimes_{W(k)} W(k) \cong M$ is an isomorphism.
\end{itemize}
By construction, the tensor product on pseudocompact $W(k)$-modules commutes with cofiltered limits and is right exact, and $*$ on torsion $W(k)$-modules commutes with filtered colimits and is left exact.
We have:
\begin{prop}\label{prop:matlismonoidal}
Let $M$, $N$ be torsion $W(k)$-modules and $K$, $L$ pseudocompact $W(k)$-modules. Then there are the following natural isomorphisms:
\begin{enumerate}
\item $I(M) \otimes I(N) \cong I(M*N)$, \label{matlismon:tensortor}
\item $I^c (K) * I^c(L) \cong I^c (K \otimes L)$, \label{matlismon:tortensor}
\item $K \otimes I(M) \cong \Ext(M,K)$, \label{matlismon:tensorext}
\item $M * I^c( K) \cong \Hom^c(K,M)$. \label{matlismon:torhom}
\end{enumerate}
\end{prop}
\begin{proof}
First note that in \eqref{matlismon:tensortor}--\eqref{matlismon:torhom}, everything commutes with filtered colimits in $M$ and $N$ and with cofiltered limits in $K$ and $L$. Thus we can without loss of generality assume that $M$, $N$, $K$, and $L$ are finite length $W(k)$-modules. In this case, $I^c = I$ is an anti-equivalence.

The universal coefficient theorem gives a natural morphism
\[
(M * I^c (K)) \otimes K \to M * (I^c(K) \otimes K),
\]
which can be composed with the evaluation $D^c K \otimes K \to CW(k)$ to give a natural map
\[
(M * I^c (K)) \to M * CW(k) \cong M.
\]
By adjunction, we obtain a map $\phi_M\colon M * I^c( K) \to \Hom(K,M)$. Since $M$ is finitely generated, it injects into an injective $J_0 \cong CW(k)^N$ for some $N\geq 0$ , the quotient $J_1 = J_0/M$ being of the same form. We get a four-term exact ladder
\[
\begin{tikzpicture}
 	\matrix (m) [matrix of math nodes, row sep=2em, column sep=2em, text height=1.5ex, text depth=0.25ex]
 	{
 		0 & M * I^c (K) & J_0 * \mathbb I^c(K)  & J_1 * \mathbb I(K) & M \otimes I^c K & 0 \\
 		0 & \Hom(K,M) & \Hom(K,J_0) & \Hom(K,J_1) & \Ext(K,M) & 0\\
 	};
 	\path[->,font=\scriptsize]
 	(m-1-1)	edge (m-1-2)
 	(m-1-2) edge (m-1-3) edge (m-2-2)
 	(m-1-3) edge (m-1-4) edge (m-2-3)
 	(m-1-4) edge (m-1-5) edge (m-2-4)
 	(m-1-5) edge (m-1-6) edge (m-2-5)
 	(m-2-1)	edge (m-2-2)
	(m-2-2) edge (m-2-3)
	(m-2-3) edge (m-2-4)
	(m-2-4) edge (m-2-5)
	(m-2-5) edge (m-2-6);
\end{tikzpicture}
\]
The map $\phi_{CW(k)}$ is trivially an isomorphism, hence so are $\phi_{J_0}$ and $\phi_{J_1}$ since $J_0$ and $J_1$ are finite sums of $CW(k)$ and both $*$ and $\Hom(K,-)$ commute with finite sums. By exactness, \eqref{matlismon:tensorext} and \eqref{matlismon:torhom} follow.

For \eqref{matlismon:tensortor}, we compute
\begin{align*}
I (M * N) &= W(k) \otimes I(M * N) \underset{\eqref{matlismon:tensorext}}\cong \Ext(M * N,W(k)) \cong \Ext(M,\Ext(N,W(k)))\\
& \underset{\eqref{matlismon:tensorext}}\cong \Ext(M,I(N) \otimes W(k) ) \underset{\eqref{matlismon:tensorext}}\cong I( M) \otimes I(N) \otimes W(k) \cong I(M) \otimes I(N).
\end{align*}
Lastly, \eqref{matlismon:tortensor} follow from \eqref{matlismon:tensortor} by dualization.
\end{proof}

A Dieudonn\'e module $M \in \DMod{k}^p$ is in particular a torsion $W(k)$-module: $M \cong M^d \oplus M^c$ where on the connected part $M^c$, $V$ is nilpotent and thus so is $p=FV$; and $M^d \cong (W(\bar{k}) \otimes_{\mathbb{Z}} L)^{\bar{k}}$ for a torsion $p$-module $L.$ Thus $I(M)$ is defined as a pseudocompact $W(k)$-module, and it inherits the structure of an $\mathcal R$-module by defining, for $\phi\colon M \to CW(k) \in I(M)$, 
\[
F(\phi) = F_{CW(k)}(\phi \circ V) \quad \text{and} \quad V(\phi) = F^{-1}_{CW(k)}(\phi \circ F)
\]
where $F_{CW(k)}$ is the Frobenius on $CW(k).$ 
Thus $I(M)$ is an object of $\DDMod{k},$ and this shows that $I$ is an anti-equivalence of categories $I: \DMod{k} \rightarrow \DDMod{k}$ with inverse $I^c:\DDMod{k} \rightarrow \DMod{k}.$  
\subsection{Cartier duality on the level of Dieudonné modules}
We now move onto proving the main theorem of this section. We start by considering unipotent groups. 
\begin{prop}\label{prop:DieudonneCommuteCartierUni}
Let $G$ be a unipotent group and let $G^*$ be its Cartier dual. Then there is a natural isomorphism $$\DieuFor(G^*) \cong I(D(G))$$ and conversely, if $H$ is a connected group and $H^*$ its Cartier dual, then $$D(H^*) \cong I^c(\DieuFor(H)).$$ 
\end{prop}
If $M \in \DDModF{k},$ we let $I^c_{F^n}(M)$ be the Dieudonné submodule of $$\Hom^c_{W(k)}(M,CW(k))$$ consisting of those $f: M \rightarrow CW(k)$ such that $f(F^n)=0.$ Let us consider the connected formal group group $W^{\fin}_k$ as in Example \ref{ex:dieudonneW}. Denote by $W^{\fin}_{k,p}$ the $k$-scheme we get after base change along the Frobenius map $F_k: \Spec k \rightarrow \Spec k.$ Then the relative Frobenius gives us a map $W^{\fin}_k \rightarrow W^{\fin}_{k,p}$ of formal groups over $k.$ However, since $$W^{\fin}_k \cong W^{\fin}_{\mathbb{F}_p} \times_{\mathbb{F}_p} \Spec k,$$ we see that $W^{\fin}_{k,p} \cong W^{\fin}_k$ as formal group over $k.$ Composing the Frobenius with this isomorphism, we get an endomorphism $F_{W^{\fin}_k}: W^{\fin}_k \rightarrow W^{\fin}_k.$ For $i \geq 0,$ write $W^{\fin}[F^i]$ for the kernel of the $i$th iterate of this Frobenius map. Note that for $i \geq 0,$ composition with $F_{W^{\fin}_k}$ gives a map $F_{W^{\fin}_k}:W^{\fin}[F^{i+1}] \rightarrow W^{\fin}[F^i].$
\begin{lemma}\label{lemma:Dieudonnerep}
Let $W^{\fin}_k[F^n] = \ker(F^n:W^{\fin}_k \rightarrow W^{\fin}_k)$ where $F^n, n \geq 1$ is the $n$th iterate of the Frobenius endomorphism, and let $\DieuFor(W^{\fin}_k[F^n])$ be its Dieudonné module. Then for $M \in \DDModF{k},$ there is a $W(k)$-linear isomorphism $$\pi_n: \Hom_{\DDModF{k}}(M,\DieuFor(W^{\fin}_k[F^n])) \rightarrow I^c_{F^n}(M).$$ Further, this isomorphism is compatible with the Frobenius in the sense that if $$f \in  \Hom_{\DDModF{k}}(M,\DieuFor(W^{\fin}_k[F^n]))$$ and $$\DieuFor(F_{W^{\fin}_k}):\DieuFor(W^{\fin}_k[F^n]) \rightarrow  \DieuFor(W^{\fin}_k[F^{n+1}]),$$ then $$\pi_{n+1}(\DieuFor(F_{W^{\fin}_k}) \circ f) = \pi_n(f).$$   
\end{lemma}
\begin{proof}
By Example \ref{ex:dieudonneW}, we see that a morphism  $$f: M \rightarrow \DieuFor(W_k^{\fin}[F^n]) \cong \DieuFor(W_k^{\fin})/F^n(\DieuFor(W_k^{\fin}))$$ can be written as $$(f_0,f_1,f_2, \ldots) $$  where $f_i$ is a continuous homomorphism of $W(k)$-modules from $M$ to $$W_{\min\{n-1,i\}}(k).$$ The requirement that $f$ is a morphism of Dieudonné modules gives that $f$ is in fact, determined by $f_{n-1}.$ Indeed, the requirement that $$f(Vm) = Vf(m)$$ gives that $$f_{n-1}(V^im) = F_{W(k)}^{-i}(f_{n-1+i}(m)), i \geq 0$$ and $f(Fm) = Ff(m)$ shows that $$f_{n-1}(F^i m)= p^i F^i_{W(k)}( f_{n-1-i}(m)), i \geq 0.$$ We see that a morphism $$f: M \rightarrow \DieuFor(W_k^{\fin}[F^n])$$ gives rise to a map $$f_{n-1}:M \rightarrow W_{n-1}(k) \subset CW(k)$$ such that $f_{n-1}(F^n)=0.$ We define $\pi_n(f) = F^{-n+1}_{W(k)} \circ f_{n-1}.$ Note that the action of $W(k)$ on $W_{n-1}(k)$ is, when we view  $W_{n-1}(k)$ as a submodule of $CW(k),$  given by for $x \in W(k)$ and $a \in W_{n-1}(k)$ by $F^{-n+1}(x) a,$ thus $\pi_n(f)$ is $W(k)$-linear. Conversely, given a $W(k)$-linear map $$h: M \rightarrow W_{n-1}(k) \subset CW(k)$$ such that $h(F^n)= 0,$ one easily defines a map $\tilde{h}: M \rightarrow \DieuFor(W^{\fin}_k[F^n])$ of Dieudonné modules. Lastly, to see that $$\pi_{n+1}(\DieuFor(F_{W_k^{\fin}}) \circ f) = \pi_n(f),$$ is a straightforward computation. 	
\end{proof}
\begin{proof}[Proof of Prop.~\ref{prop:DieudonneCommuteCartierUni}]
Let $G$ be as in the proposition. We can without loss of generality assume that $V^m_G=0.$  Consider $$D(G) \cong \colim_n \Hom(G,W^n_k) \cong \colim_n \Hom(W^{\fin}_k[F^n],G^*),$$ where the latter isomorphism is induced by Cartier duality. Note that the transition maps $\Hom(W^{\fin}_k[F^n],G^*) \rightarrow \Hom(W^{\fin}_k[F^{n+1}],G^*)$ come from precomposition with the Frobenius on $W_k^{\fin}.$ By classical Dieudonné theory, $$ \Hom_{\FGpsc{k}}(W_k^{\fin}[F^n],G^*) \cong \Hom_{\DDModF{k}}(\DieuFor(G^*),\DieuFor(W_k^{\fin}[F^n])).$$ By Lemma \ref{lemma:Dieudonnerep}, $\Hom_{\DDModF{k}}(\DieuFor(G^*),\DieuFor(W_k^{\fin}[F^n]))$ is isomorphic to $I^c_{F^n}(\DieuFor(G^*)).$  Taking the limit over all $n,$ using Lemma \ref{lemma:Dieudonnerep} to see that the diagram $$\begin{tikzcd} \Hom(\DieuFor(G^*), \DieuFor(W_k^{\fin}[F^n])) \arrow[r,"F_{\DieuFor(W_k^{\fin})}"] \arrow[d, "\pi_n"] & \Hom(\DieuFor(G^*),\DieuFor(W_k^{\fin}[F^{n+1}])) \arrow[d, "\pi_{n+1}"] \\ I^c_{F^n}(M) \arrow[r,"i_n"] & I^c_{F^{n+1}}(M) \end{tikzcd}$$ commutes, one gets that $$D(G) \cong I^c(\DieuFor(G^*))$$ as $W(k)$-modules. A verification, keeping track of the interchange between the Frobenius and Verschiebung when we dualize, shows that the isomorphism $$D(G) \cong I^c(\DieuFor(G^*))$$ actually is an isomorphism of Dieudonné modules.
\end{proof}
\begin{proof}[Proof of Thm.~\ref{thm:CartierMatlis}]
We have already proved this for unipotent groups by Prop.~\ref{prop:DieudonneCommuteCartierUni}, so what remains is the proof for $p$-adic multiplicative groups. If $G$ is a p-adic multiplicative group, then $G$ is an inverse limit of finite groups of $p$-power order. We can thus assume that $G$ is a finite multiplicative group of order $p^n.$ But then this duality is classical, see for example \cite[III,\textsection 5]{fontaine:groupes-divisibles}. 
\end{proof}

\section{Tensor products of formal Hopf algebras in positive characteristic} \label{sec:tensorprodHopf}

\subsection{Tensor products of formal Hopf algebras with a lift of the Verschiebung}

Let $\mathcal{C}_V$ be the category of topologically free coalgebras over $W(k)$ together with a lift of the Verschiebung. We let $\mathcal{H}_V \cong (\FgpsV)^{\op}$ be the category of complete Hopf algebras over $W(k)$ representing connected, smooth formal groups. The purpose of this section is to first show that there is a tensor product $\boxtimes$ in $\mathcal{H}_V,$ and then give a formula for the tensor product of two objects.

\begin{prop}
The category $\mathcal H_V$ has a tensor product $\boxtimes^c$.
\end{prop}
\begin{proof}
We apply Theorem~\ref{thm:goerss} to $\mathcal{C}_V$ and $\mathcal{H}_V$. The categorical product of $C_1$ and $C_2 \in \mathcal{C}_V$ is given by the tensor product $C_1 \hat\otimes_{W(k)} C_2$ with the diagonal $V$-action.To show that $\mathcal{H}_V$ has coequalizers, one notes that if $f,g\colon H_1 \rightarrow H_2$ are a pair of maps in $\mathcal{H}_V,$ then $\coeq(f,g)$ is defined as the minimal smooth quotient (minimal in the sense that any other smooth quotient factors through it) of $H_2/I,$ where $I$ is the closure of the ideal $(f(h)-g(h);\; h \in H_1)$. 

Finally, the forgetful functor $\mathcal{H}_V \rightarrow \mathcal{C}_V$ has a left adjoint
\[
S_*\colon \mathcal{C}_V \rightarrow \mathcal{H}_V, \quad C \mapsto W(k)\pow{J(C)},
\]
where $JC$ is the augmentation ideal and the coalgebra structure is induced from $C$. To see that $S_*(C)$ has an antipode, consider the algebra $\End_{\mathcal M_V}(S_*(C))$ of linear maps commuting with $V$, where multiplication is given by $(fg) = \sum_{(x)} f(x')g(x'')$. The unity in this ring is $1=\eta\circ\epsilon$, and an antipode is an inverse for the identity map in this algebra. As in \cite[Lemma~9.2.3]{sweedler:hopf-algebras}, the antipode is given by
\[
\id^{-1} = \frac 1{1-(1-\id)} = \sum_{i \geq 0} (1-\id)^i,
\]
a pointwise (even uniformly) convergent power series in $\End_{\mathcal M_V}(S_*(C))$.
\end{proof}

By Thm.~\ref{thm:indecomposableequiv}, any $H \in \mathcal{H}_V$ is isomorphic to the representing object of some formal group $\FrM(M)$ with a Verschiebung lift for some $M \in \mathcal{M}_V.$  We will then write $H = S^*(M).$

\begin{lemma}[{\cite[Proposition ~6.1]{goerss:hopf-rings}}]  \label{lemma:goersslemma}
The functor $S^*\colon (\mathcal M_V,\hat\otimes_{W(k)}) \to (\mathcal H_V,\boxtimes^c)$ is strongly monoidal.
\end{lemma}
\begin{proof}
This proof is identical to the proof by Goerss, but we include it for the the reader's convenience. If $U\colon \mathcal H_V \to \mathcal C_V$ denotes the forgetful functor and $J\colon\mathcal{C}_V \rightarrow \mathcal{M}_V$ the coaugmentation ideal, we have a pair of adjoint diagrams of functors
\[
\begin{tikzcd}
\mathcal H_V \arrow[dr,swap,"Q"] & \mathcal C_V \arrow[l,swap,"S_*"] \arrow[d,"J"] \\ & \mathcal M_V.
\end{tikzcd}
\quad \text{left adjoint to} \quad
\begin{tikzcd}
\mathcal H_V \arrow[r,"U"] & \mathcal C_V\\ & \mathcal M_V \ar[lu,"S^*"] \ar[u]
\end{tikzcd}
\]
Since the right hand diagram commutes, the right adjoint of $J$ is $US^*$. We will suppress the forgetful functor $U$ from notation in what follows.

Now let $\phi\colon S^*(M_1) \otimes S^*(M_2) \rightarrow S^*(K)$ be a bilinear map in $\mathcal C_V$; explicitly and in Sweedler notation,
\begin{align*}
\phi(xy,z) &= \sum_{(z)} \phi(x,z^{(1)}) \phi(y,z^{(2)})\\
\phi(x,zw) &= \sum_{(x)} \phi(x^{(1)},z) \phi(x^{(2)},w).
\end{align*}
It has a left adjoint in $\mathcal M_V$,
\[
J(S^*(M_1) \otimes S^*(M_2)) \rightarrow K
\]
which factors as
\[
J(S^*(M_1) \otimes S^*(M_2)) \rightarrow J(S^*(M_1)) \otimes J(S^*(M_2)) \rightarrow M_1 \otimes M_2 \xrightarrow{Q \phi} K.
\]
Let $\eta\colon S^*(M_1) \otimes S^*(M_2) \rightarrow S^*(M_1 \otimes M_2)$ be the adjoint of the composite of the first two maps. If $\eta$ is bilinear, then $S^*(M_1 \otimes M_2)$ has the required universal property, and the theorem is proven.

To show this, take $C \in \mathcal C_V$ and consider the bilinear map 
\begin{align*}
& \Hom_{\mathcal{M}_V}(JC,M_1) \times \Hom_{\mathcal{M}_V}(JC,M_2) \to \Hom_{\mathcal M_V}(JC \otimes JC,M_1 \otimes M_2)\\
& \to \Hom_{\mathcal M_V}(J(C \otimes C),M_1 \otimes M_2) \xrightarrow{\Delta_C^*} \Hom_{\mathcal{M}_V}(JC,M_1 \otimes M_2).
\end{align*}
By naturality and adjunction, we get a bilinear pairing
\[
\Hom_{\mathcal{H}_V}(-, S^*(M_1)) \times \Hom_{\mathcal{H}_V}(-,S^*(M_2)) \rightarrow \Hom_{\mathcal{H}_V}(-,S^*(M_1 \otimes M_2))
\]
and thus by Yoneda a bilinear map 
\[
S^*(M_1) \otimes S^*(M_2) \rightarrow S^*(M_1 \otimes M_2).
\]
This map agrees with $\eta$  since they both are adjoint to $$J(S^*(M_1) \otimes S^*(M_2)) \rightarrow JS^*(M_1) \otimes JS^*(M_2) \rightarrow M \otimes N.$$ This implies that $\eta$ is bilinear. 
\end{proof}
\begin{corollary}[{\cite[Cor.~6.2]{goerss:hopf-rings}}] \label{cor:goersscorollary}
Let $H_i = S^*(M_i)\in \mathcal{H}_V$ ($i=1,2$). Then the universal bilinear map $\eta\colon H_1 \otimes H_2 \rightarrow H_1 \boxtimes^c H_2$ induces an isomorphism in $\mathcal M_V$,
\[
Q\eta: M_1 \otimes M_2 \rightarrow Q(H_1 \boxtimes^c H_2).
\]
\end{corollary}
\begin{proof}
By bilinearity of $\eta$, applying $Q$ gives the map $Q \eta : M_1 \otimes M_2 \rightarrow Q(H_1 \boxtimes^c H_2) \cong Q(S^*(M_1 \otimes M_2)).$  We have that both the counit map $Q(H_1 \boxtimes^c H_2) \cong Q(S^*(M_1 \otimes M_2)) \rightarrow M_1 \otimes M_2$ and the composition $$M_1 \otimes M_2 \xrightarrow{Q \eta} Q(H_1 \boxtimes^c H_2) \rightarrow M_1 \otimes M_2$$ are isomorphisms, thus $Q \eta$ is an isomorphism as well. 
\end{proof}
\begin{example}\label{ex:CWtensor}
Let $H=\Reg{CW^{u,c}_{W(k)}} \in \mathcal H_V$ corepresent the functor $CW^{u,c}_{W(k)}$ from in Section \ref{sec:smoothFormal}. Then $H \boxtimes^c H$ corepresents the functor 
\[
\bigoplus_{i=-\infty}^{\infty} CW^{u,c}_{W(k)}
\]
with the obvious lift of the Verschiebung. Indeed, the formal Hopf algebra $H \boxtimes^c H$ has indecomposables
\[
QH \hat{\otimes}_{W(k)} QH \cong \Bigl(\prod_{i=0}^\infty W(k)\Bigr) \hat{\otimes}_{W(k)} \Bigl(\prod_{i=0}^\infty W(k)\Bigr) \cong \prod_{i,j=0}^\infty W(k),
\] where $(Vx)_{i,j} = F^{-1}_{W(k)}x_{i+1,j+1}$. The claim follows by Lemma~\ref{lemma:goersslemma} and Thm.~\ref{thm:indecomposableequiv}. 
\end{example}

\subsection{Tensor products of general formal Hopf algebras}

Let $\FHopf$ be the category of formal Hopf algebras over a perfect field $k$ of characteristic $p>0$. We denote by $\FCoalg{k}$ the category of formal coalgebras over $k.$ By \eqref{eq:antiequivalences}, $\FHopf$ is equivalent to the category of affine groups, which has tensor product by Thm.~\ref{thm:existence}. Denote the induced tensor product on $\FHopf$ by $\boxtimes$. Thus there is a bijective correspondence between maps $H_1 \boxtimes H_2 \rightarrow K$ of formal Hopf algebras and bilinear morphisms  $H_1 \hat{\otimes}_k H_2 \rightarrow K$ of formal coalgebras. 

\begin{remark} \label{rmk:structureoftensor}
For an explicit description, let $H_1, H_2 \in \FHopf$ and $\hat{S}_*(H_1 \hat{\otimes}_k H_2)$ be the completed symmetric algebra (see  \cite[Exposé VII, 1.2.5]{Gille:SGA3}) on $H_1 \hat{\otimes}_k H_2$. Let $j\colon H_1 \hat{\otimes}_k H_2 \rightarrow \hat{S}_*(H_1 \hat{\otimes}_k H_2)$ be the inclusion.
Then $H_1 \boxtimes H_2$ is the quotient of $\hat{S}_*(H_1 \hat{\otimes}_k H_2)$ by the closure of the ideal generated by the elements
\begin{align*}
j(xy \otimes z)&-\sum_{(z)} j(x,z^{(1)}) j(y,z^{(2)}),\\
j(x \otimes zw)&- \sum_{(x)} j(x^{(1)},z)j(x^{(2)},w),\\
j(x \otimes 1) &- \epsilon_{H_1}(x),\\
j(1 \otimes y) &- \epsilon_{H_2}(y),
\end{align*}
where we used Sweedler's notation. The coaddition on $H_1 \boxtimes H_2$ is the one making the map $j: H_1 \hat{\otimes}_k H_2 \rightarrow H_1 \boxtimes H_2$ a morphism of formal coalgebras. The connected part of $H_1 \boxtimes H_2,$ is a formal power series ring in $JH_1^c \hat{\otimes}_k JH^c_2,$ where $JH_i^c$ is the augmentation ideal of the connected part, modulo the closure of the ideal generated by the elements above.
\end{remark}

With respect to the splitting \eqref{eq:antiequivalenceswithsplitting}, the tensor product behaves as follows:
\begin{lemma}\label{lemma:connectedtensorconnectednotconnected}
The tensor product of an étale formal Hopf algebra with any formal Hopf algebra is étale. The tensor product of two connected formal Hopf algebras $H_1$, $H_2$ splits naturally into a connected and an étale part,
\[
H_1 \boxtimes H_2 \cong H_1 \boxtimes^c H_2 \;\; \hat\otimes\;\; H_1 \boxtimes^e H_2,
\]
both of which are nontrivial in general.
\end{lemma}
\begin{proof}
The first part is equivalent to Proposition ~\ref{prop:multiplicativetensor}. To illustrate that the tensor product of two connected formal Hopf algebras need not be connected, let $H_1 = H_2$ be the primitively generated (finite, hence formal) Hopf algebra $k[x]/(x^p)$ representing $\alpha_p,$ the formal group taking a finite $k$-algebra $A$ to its kernel of the $p$th power map. 
We then claim that $H_1 \boxtimes^e H_2$ is nontrivial, i.e that there is a non-zero morphism $\underline{\mathbb{Z}} \to \Spf {(H_1 \boxtimes H_2)}$ from the constant formal group $\underline{\Z}$.

By Cartier duality ($\alpha_p$ is self-dual), this is equivalent to having a non-trivial map $\alpha_p \otimes \alpha_p \rightarrow \mathbb{G}_{m},$ or by adjunction, a non-trivial map
\[
\alpha_p \rightarrow \underline{\Hom}(\alpha_p,\mathbb{G}_{m}) \cong \alpha_p \quad \text{(Cartier duality)}.
\]
The identity is such a map.
\end{proof}

In the following two subsections, we will first describe $H_1 \boxtimes^c H_2$ (heavily inspired by \cite{goerss:hopf-rings}), and then $H_1 \boxtimes^e H_2$.

\subsection{The connected part of $\boxtimes$ for two connected formal Hopf algebras}

By abuse of notation, we will denote the connected covariant Dieudonné functor $\FHopfc \rightarrow \DDModF{k}$,
\[
H \mapsto \DieuFor(\Spf H) = \Hom_{\FHopfc}(\Reg{CW^c_k}, H), \quad \text{(cf. Thm.~\ref{thm:dieudonneforformalgroups})}
\]
by $\DieuFor$ as well.

We start by describing $\DieuFor(H_1 \boxtimes^c H_2)$ for $H_1,H_2 \in \FHopfc$ in the situation where both $H_1$ and $H_2$ are reductions mod $p$ of formal Hopf algebras with a Verschiebung lift over $W(k)$.
\begin{lemma}[{\cite[Lemma 7.1]{goerss:hopf-rings}}] \label{lemma:boxtimescommuteswithbasechange} 
For $H_1,H_2 \in \mathcal{H}_V$, the canonical map $$(k \otimes H_1) \boxtimes^c (k \otimes H_2) \rightarrow k \otimes (H_1 \boxtimes^c H_2)$$ is an isomorphism.
\end{lemma}
\begin{proof}
Both the source and the target of the map represent connected formal groups. Since the target is a smooth formal group, it suffices that it induces an isomorphism on indecomposables. By the explicit construction of $H_1 \boxtimes^c H_2$ as detailed in Remark \ref{rmk:structureoftensor}, we see that if $\{x_i\}_{i \in I} \in Q(H_1)$ and $\{y_j\}_{j \in J} \in Q(H_2)$ are topological bases, then $\{x_i \boxtimes y_j\}_{i \in I, j \in J}$ topologically generates $Q((k \otimes H_1) \boxtimes^c (k \otimes H_2))$. By Cor.~\ref{cor:goersscorollary}, the elements $\{x_i \boxtimes y_j\}_{i \in I,j \in J}$ constitute a topological basis of $Q(k \otimes (H_1 \boxtimes^c H_2)).$ 
\end{proof}
\begin{corollary} \label{cor:dieuboxforverschiebunglift}
Let $H_1,H_2 \in \mathcal{H}_V.$ Then $$\DieuFor((k \otimes H_1) \boxtimes^c (k \otimes H_2)) \cong \mathcal{R} \hat{\otimes}^F_{\mathcal V} (QH_1 \hat{\otimes}_{W(k)}QH_2).$$ 
\end{corollary}
\begin{proof}
This follows directly from Thm.~\ref{thm:indecomposablesdieudonne} and Lemma \ref{lemma:boxtimescommuteswithbasechange}.
\end{proof}
Given a bilinear map $H_1 \otimes H_2 \rightarrow K$ in $\FHopfc,$  there is a natural function $$\DieuFor(H_1) \otimes_{W(k)} \DieuFor(H_2) \rightarrow \DieuFor(K)$$ satisfying the following axioms:
\begin{defn}
Let $M_1,M_2,N \in \DDModF{k}.$ We say that a function $\varphi: M_1 \times M_2 \rightarrow N$ is a \emph{Dieudonné pairing} if  $\varphi$ is $W(k)$-bilinear, continuous, and satisfies \begin{enumerate} 
\item $\varphi(Vx,Vy) = V \varphi(x,y) $
\item $\varphi(Fx,y) = F \varphi(x,Vy)$ 
\item $\varphi(x,Fy) = F \varphi(Vx,y)$
for all $x \in M_1,y \in M_2.$ 
\end{enumerate}
We denote by $\DieuP(M_1,M_2;N)$ the group of Dieudonné pairings $M_1 \times M_2 \rightarrow N.$
\end{defn}

Let $CW^c_k$ be the connected co-Witt vector functor of Section~\ref{subsec:backgroundWitt} and $C_k=\Reg{CW^c_k}$ its corepresenting formal Hopf algebra. Then $C_k \cong k \otimes C_{W(k)}$ for $C_{W(k)}=\Reg{CW^c_{W(k)}} \in \mathcal H_V$. Recall that $C_{W(k)}$ is a suitable completion of the polynomial ring $W(k)[x_0,x_{-1},\dots]$ in countable many indeterminates. Let $\eta \in Q(C_k)$ be the equivalence class of $x_0 \in C_k$ and $\iota\colon C_k \to C_k \boxtimes^c C_k$ be the map corresponding to the element
\[
1 \otimes \eta \otimes \eta \in \mathcal{R} \hat{\otimes}^F_{\mathcal V} (QC_k \hat{\otimes}_{W(k)}QC_k) \underset{\text{Cor.~\ref{cor:dieuboxforverschiebunglift}}}\cong \DieuFor(C_k \boxtimes^c C_k).
\]
Any bilinear map $f\colon H_1 \otimes H_2 \rightarrow K$ of formal connected Hopf algebras induces a homomorphism $\tilde f\colon H_1 \boxtimes^c H_2 \rightarrow K$. We define the function $\mu_f\colon \DieuFor(H_1) \times \DieuFor(H_2) \rightarrow \DieuFor(K)$ as the composite
\begin{align}
\mu_f\colon & \DieuFor(H_1) \times \DieuFor(H_2) =  \Hom_{\FHopfc}(C_k,H_1) \times \Hom_{\FHopfc}(C_k,H_2) \label{eq:muf}\\
& \xrightarrow{\boxtimes^c}  \Hom_{\FHopfc} (C_k \boxtimes^c C_k, H_1 \boxtimes^c H_2) \xrightarrow{\iota^*}  \Hom_{\FHopfc} (C_k, H_1 \boxtimes^c H_2) \notag \\
& \xrightarrow{\tilde{f}_*}  \Hom_{\FHopfc}(C_k,K) = \DieuFor(K).\notag
\end{align}

We now want to prove that $\mu_f$ is a Dieudonné pairing.

If $H$ is a formal connected Hopf algebra, $\DieuFor(H)$ is a subspace of $CW^c_k(H)$ and thus inherits a natural topology. For any finite $k$-algebra $A,$ we now want to equip $(\Spf{C_k \boxtimes^c C_k})(A) = \Hom^c_{\Alg_{k}}(C_k \boxtimes^c C_k,A)$ with a natural topology.
For a $k$-module $M$, denote by $M^\infty \subset M^{\N} \twoheadleftarrow M^{(\infty)}$ its infinite direct sum, infinite product, and the profinite completion of its infinite direct sum, respectively. 

By Lemma \ref{lemma:boxtimescommuteswithbasechange}, we know that
\[
C_k \boxtimes^c C_k \cong k\pow{k^{(\infty)} \hat{\otimes}_{k} k^{(\infty)}}.
\]
 Thus $Q(C_k \boxtimes^c C_k) \cong k^{(\infty)} \hat{\otimes}_{k} k^{(\infty)}$. This is \emph{not} isomorphic to the profinite completion of $k^\infty \otimes k^\infty,$ but to the completion of the latter with respect to  the submodules $M \otimes k^\infty + k^\infty \otimes N$, where $M,\; N \subset k^\infty$ are submodules such that the quotients $k^\infty/M$ and $k^\infty/N$ are finite length $k$-modules.

 For a finite $k$-algebra $A,$ we have an inclusion
 \[
 \Hom^c_{\Alg_{k}}(C_k \boxtimes^c C_k,A)  \subset \Nil(A)^{\N \times \N}.
 \]
 We use this inclusion to topologize $\Hom^c_{\Alg_{k}}(C_k \boxtimes^c C_k,A)$ as a subspace of $\Nil(A)^{\N \times \N}$. This definition extends to a natural topology on both $\Hom^c_{\Alg_{k}}(C_k \boxtimes^c C_k,A)$ for any profinite $k$-algebra $A,$ and $\Hom_{\FHopfc}(C_k \boxtimes^c C_k, H)$ for any connected formal Hopf algebra~$H.$ 
\begin{lemma} \label{lemma:iotacont}
For any pseudocompact $k$-algebra $A,$ the morphism $$\iota^*: \Spf{\Reg{CW^c_k} \boxtimes^c \Reg{CW^c_k}}(A) \rightarrow CW^c_k(A)$$ is continuous. 
\end{lemma}
To prove this lemma, we need another definition of $\iota.$ The category of (cocommutative) torsion-free Hopf algebras over $W(k)$ with a Verschiebung lift has a tensor product, studied in the graded case by Goerss \cite{goerss:hopf-rings}.  We denote this tensor product by $\boxtimes_a$ in order not to confuse it with our tensor product. If we let $\Reg{W^n_{W(k)}}$ be the Hopf algebra representing the functor taking an algebra over $W(k)$ to its length $n$ Witt vectors, then $\Reg{W^n_{W(k)}} \cong W(k)[x_0, \ldots,x_{n-1}],$ where the coaddition is given by the addition for Witt vectors. Give this ring the standard grading with $|x_i|=p^i$. We have an isomorphism of bigraded Hopf algebras with a Verschiebung lift,
\[
\Reg{W^n_{W(k)}} \boxtimes_a \Reg{W^n_{W(k)}}\cong  W(k)[ (x_{ij})_{0 \leq i,j \leq n}],
\]
where $V$ acts diagonally and $x_{ij}=x_i \boxtimes x_j$ in bidegree $(p^i,p^j)$. 
\begin{prop}[{\cite[Corollary 1.2.21]{hopkins-lurie:ambidexterity}}] \label{prop:HLiota} 
There exists a unique Hopf algebra map $$\iota_n: \Reg{W^n_{W(k)}} \rightarrow \Reg{W^n_{W(k)}} \boxtimes_a \Reg{W^n_{W(k)}}$$ which takes the $w_n,$ the $n$th ghost polynomial, to $\dfrac{w_n \boxtimes w_n}{p^n}$. Furthermore, for each $n,$ this map is compatible with the Verschiebung maps in the obvious sense. 
\end{prop}

The map $\iota_n$ maps elements of degree $p^i$ to elements of bidegree $(p^i,p^i).$ For each $n,$ we get a homogeneous polynomial $$P_n((x_{ij})_{0 \leq i,j \leq n-1}) := \iota_n(x_{n-1}) \in W(k)[(x_{ij})_{0 \leq i,j \leq n}]$$ of bidegree $(p^{n-1},p^{n-1}).$ We think of $P_n$ as a polynomial in the entries of an $(n+1)\times(n+1)$ square matrix, with $x_{00}$ in the top left entry and $x_{n0}$ in the bottom left entry.

From the formula from Prop.~\ref{prop:HLiota},\begin{equation}
P_n((x_{ij})_{0 \leq i,j \leq n-1}) = x_{n-1,n-1} + Q((x_{i,j})_{0 \leq i,j \leq n-2}) \label{eq:pn}
\end{equation}
for some bihomogeneous polynomial $Q$ of bidegree $(p^{n-1},p^{n-1}).$ The proof of the following lemma is almost verbatim the proof of \cite[Lemme~II.1.3]{fontaine:groupes-divisibles}. 

\begin{lemma} \label{lemma:polyconverge}
In the ring $W(k)[(x_{-i,-j})_{i,j \in \mathbb{N}}],$ define the ideal
\[
J_{r,s} = (x_{-i,-j} \mid i \geq r,\; j \geq s) \quad (r,s \geq 0).
\]
Then
\[
P_{n+1}((x_{i,j})_{-n \leq i,j \leq 0} ) = P_n( (x_{i,j})_{-n+1 \leq i,j \leq 0}) \pmod{J_{1,r}^s+J_{r,1}^s}
\]
for each integer \begin{equation*} n \geq \begin{cases}  r & \text{if } s < p \\ r+(s-p)/(p-1) & \text{if } s \geq p.  \end{cases} \end{equation*}
\end{lemma}
\begin{proof}
From the fact that the $\iota_n$ commute with the Verschiebung, we see that in $$\Reg{W^{n}_{W(k)}} \boxtimes_a \Reg{W^{n}_{W(k)}},$$  $P_{n+1}((x_{i-1,j-1})_{0 \leq i,j \leq {n}}) = P_n((x_{i,j})_{0 \leq i,j \leq n-1})$ with the understanding that $x_{i,j} = 0$ if $i<0$ or $j<0$. This, together with the homogeneity of the polynomials and Formula \ref{eq:pn}, implies that
\[
P_{n+1}((x_{i,j})_{-n \leq i,j \leq 0} ) - P_n( (x_{i,j})_{-n+1  \leq i,j \leq 0})
\]
is equal to a linear combination of terms of the form $$\prod_{i,j=1}^{n} (x_{-i,-j})^{u_{ij}},$$ for some exponents $u_{ij} \geq 0$. Writing $v_i = \sum_{j=1}^n u_{ij}$ and $w_j = \sum_{i=1}^n u_{ij},$ bihomogeneity implies that $\sum_i p^{n-i} v_i = p^n = \sum_k p^{n-j}w_j$. Either $v_n \neq 0$ or $w_n \neq 0.$ Suppose that $v_n \neq 0.$ We have $$v_n+pv_{n-1} + \cdots + p^{n-1}v_1 = p^n.$$ We see that for each $0 \leq t  < n,$  $v_n+pv_{n-1} + \cdots + p^{n-t}v_{n-t}$ is divisible by $p^{t+1}.$ By \cite[Lemme 1.2]{fontaine:groupes-divisibles} this implies that $v_n + v_{n-1} + \cdots + v_{n-t} \geq t(p-1)+p,$ and then the term lies in $J^{t(p-1)+p}_{n-t,1}.$ Similarily, if $w_n \neq 0$ we see that the term lies in $J^{t(p-1)+p}_{1,n-t},$ so $$P_{n+1}-P_n \in J^{t(p-1)+p}_{1,n-t} + J^{t(p-1)+p}_{n-t,1}.$$ If $t=0,$ the difference lies in $J_{n,1}^p+J_{1,n}^p,$ which gives the lemma if $s < p.$ If $ s \geq p,$ and if $ n \geq r+(s-p)/(p-1),$ by letting $t = n-r,$ $J_{r,1}^{t(p-1)+p}+J_{1,r}^{t(p-1)+p} \subset J_{r,1}^s+J_{1,r}^s$ since $t(p-1)+p \geq s.$
\end{proof}
We thus get a sequence of polynomials $$P_{n+1} ((x_{i,j})_{-n \leq i,j \leq 0}) \in \Reg{CW^c_{W(k)}} \boxtimes \Reg{CW^c_{W(k)}}$$ which by Lemma \ref{lemma:polyconverge} converges to a power series $P((x_{i,j})_{i,j\leq0})).$ We now define the map 
\[
\iota'\colon C_{W(k)} \rightarrow C_{W(k)} \boxtimes C_{W(k)}
\]
by letting $\iota'(x_{n}) = P((x_{i,j})_{i,j\leq n})$ for $n \leq 0$ and extending by continuity. Taking into consideration how the coadditions on $C_{W(k)}$ and $C_{W(k)} \boxtimes C_{W(k)}$ are defined, one sees that this is a map of formal Hopf algebras over $W(k)$ which commutes with the Verschiebung. Since $Q\iota'(x_0) = x_0 \boxtimes x_0,$ Thm.~\ref{thm:indecomposableequiv} shows that the base change $\iota'_k$ of $\iota'$ to $k$ coincides with $\iota.$ 

\begin{proof}[Proof of Lemma~\ref{lemma:iotacont}]
We may assume that $A$ is a finite $k$-algebra. Then the map $$\iota\colon C_k \rightarrow C_k \boxtimes C_k$$  takes $x_{n}$ (where $n \leq 0$) to the limit of the polynomials $P_n((x_{n+i,n+j})_{i,j \leq 0}).$  Lemma~\ref{lemma:polyconverge} implies that $\iota$ induces a continuous homomorphism on functors of points. Indeed, there is a basis of neighborhoods of zero in $\Reg{CW^{c}_k}(A) \cong \Nil(A)^{\N}$ given by 
\[
U_k = \{(a_i)_{i \leq 0} \in A^\N \mid a_i = 0 \text{ for } a\geq -k\}.
\]
Since $\Nil(A)^{m}=0$ for some $m>0,$ Lemma \ref{lemma:polyconverge} shows that one can find an open neighborhood of $\Nil(A)^{\mathbb{N} \times \mathbb{N}}$  such that $$\Nil(A)^{\mathbb{N} \times \mathbb{N}} \cap \Hom^c_{\Alg_k}(C_k \boxtimes C_k,A)$$ maps into $U_k$ under $\iota^*.$ 
\end{proof} 
\begin{corollary} \label{cor:contofcomp}
Let $f\colon H_1 \hat{\otimes}_k H_2 \rightarrow K$ be a bilinear map in $\FHopfc.$ Then the map $$\mu_f: \DieuFor(H_1) \times \DieuFor(H_2) \rightarrow \DieuFor(K)$$ is continuous. 
\end{corollary}
\begin{proof}
We topologize $\DieuFor(H_i)$ as a closed subset of $CW^c(H_i)$ for $i=1,2$ and similarily, we topologize $\Hom_{\FHopfc}(C_k \boxtimes^c C_k, H_1 \boxtimes^c H_2)$ as a closed subset of $\Hom_{\Alg_k}(C_k \boxtimes^c C_k,H_1 \boxtimes^c H_2).$ We will show that $\mu_f$ is continuous by showing that all three maps in \eqref{eq:muf} are continuous. It is obvious that 
\[
\DieuFor(H_1) \times \DieuFor(H_2) \rightarrow \Hom_{\FHopfc}(C_k \boxtimes^c C_k, H_1 \boxtimes^c H_2)
\]
and
\[
f_*\colon\DieuFor(H_1 \boxtimes^c H_2) \rightarrow \DieuFor(K)
\]
is immediate, while the continuity of the map 
\[
\iota^*\colon \Hom_{\FHopfc}(C_k \boxtimes^c C_k, H_1 \boxtimes^c H_2) \rightarrow \DieuFor(H_1 \boxtimes^c H_2)
\]
follows from Lemma \ref{lemma:iotacont}. 
\end{proof}

\begin{lemma} \label{lemma:bilineargivespairing}
Given $H_1,H_2,K \in \FHopfc$ and a bilinear map $H_1 \otimes H_2 \rightarrow K,$ the continuous map $\mu_f: \DieuFor(H_1) \times \DieuFor(H_2) \rightarrow \DieuFor(K)$ is a Dieudonné pairing.
\end{lemma}
\begin{proof}
This proof uses ideas from \cite{hopkins-lurie:ambidexterity} and \cite{goerss:hopf-rings}. We start by verifying that $\mu_f$ is $W(k)$-bilinear. Additivity in either variable of $\mu_f$ can be shown by direct calculation or by mimicking the proof of \cite[Lemma 7.4]{goerss:hopf-rings}.
Let 
\[
x \in \DieuFor(H_1) = \Hom_{\FHopfc}(C_k,H_1) \quad \text{and} \quad y \in \DieuFor(H_2) = \Hom_{\FHopfc}(C_k,H_2).
\]
Given $a \in W(k),$ we now prove that $\mu_f(ax,y) = a\mu_f(x,y)$; linearity in the other variable follows by symmetry. The claim translates to the following diagram being commutative:
\[
\begin{tikzcd} C_k \arrow[d,"a"] \arrow[r,"\iota"] & C_k \boxtimes^c C_k \arrow[r, "a \boxtimes 1"] \arrow[d,"a \boxtimes 1"]  &  C_k \boxtimes^c C_k  \arrow[r,"x \boxtimes y"] \arrow[d, "x \boxtimes y"] & H_1 \boxtimes^c H_2 \arrow[d,"f"]   \\  C_k \arrow[r,"\iota"]& C_k \boxtimes^c C_k \arrow[r,"x \boxtimes y"] & H_1 \boxtimes^c H_2 \arrow[r,"f"] & K \end{tikzcd}
\]
Since the middle and rightmost squares commute for trivial reasons, we are reduced to showing that $(a \boxtimes 1) \iota = \iota \circ a ,$ but this follows from the relation $$a(1 \otimes \eta \otimes \eta ) = 1 \otimes a \eta \otimes \eta$$ in $\DieuFor(C_k \boxtimes^c C_k).$  In the same way, to show that $\mu_f(Vx,Vy) = V \mu_f(x,y)$ one reduces to the universal case and sees that one needs to show that $$\begin{tikzcd} C_k \arrow[d,"V"] \arrow[r,"\iota"] & C_k \boxtimes^c C_k \arrow[d,"V \boxtimes V"] \\ C_k \arrow[r,"\iota"] & C_k \boxtimes^c C_k \end{tikzcd}$$ commutes. But this is immediate from the fact that $V \boxtimes V$ is the Verschiebung on $C_k \boxtimes^c C_k$ and $\iota$ is a morphism of formal Hopf algebras. To see that the Verschiebung on  $C_k \boxtimes^c C_k$ is $V \boxtimes V,$ one argues either as in \cite[Proposition 1.3.27]{hopkins-lurie:ambidexterity} or notes that it is a consequence of Lemma \ref{lemma:goersslemma} and Lemma \ref{lemma:boxtimescommuteswithbasechange}.

We now move on to proving that
\[
\mu_f(Fx,y) = F \mu_f(x,Vy) \quad \text{and} \quad \mu_f(x,Fy) = F \mu_f(Vx,y).
\]
By symmetry, it is enough to prove the first identity. Reducing to the universal caseand denoting by $F_{C_k \boxtimes^c C_k}$ the Frobenius on $C_k \boxtimes^c C_k,$ it is enough to show that 
\[
\begin{tikzcd} C_k \arrow[r,"\iota"] \arrow[d,"F"] & C_k \boxtimes^c C_k \arrow[r,"F \boxtimes 1"] \arrow[d,"F_{C_k \boxtimes^c C_k}"] & C_k \boxtimes^c C_k  \arrow[d,equals] \\ C_k \arrow[r,"\iota"] & C_k \boxtimes^c C_k \arrow[r,"1 \boxtimes V"] & C_k \boxtimes^c C_k \end{tikzcd}
\]
commutes. The first square commutes because $\iota$ is a homomorphism of Hopf algebras, and the second commutes by Remark \ref{rmk:structureoftensor}. Indeed, with the notation of Remark \ref{rmk:structureoftensor}, the image of the coalgebra map $j:  C_k \otimes  C_k \rightarrow  C_k \boxtimes^c  C_k$ topologically generates $ C_k \boxtimes^c  C_k$ as an algebra, and we then claim that $j(Fx,y) = Fj(x,Vy).$ This equality holds since Example \ref{ex:CWtensor} tells us that the Verschiebung of $C_k \boxtimes^c C_k$ is an injective endomorphism (the associated map of formal groups is surjective), which implies that is enough to check this equality after applying $V.$ Since any coalgebra map commutes with the Verschiebung, this gives us $$Vj(Fx,y) = j(px,Vy) = pj(x,Vy)=VFj(x,Vy),$$ so the claimed equality follows.
\end{proof}
In particular, Lemma \ref{lemma:bilineargivespairing} gives a pairing $$\DieuFor(H_1) \times \DieuFor(H_2) \rightarrow \DieuFor(H_1 \boxtimes^c H_2)$$ from the universal bilinear map $H_1 \times H_2 \rightarrow H_1 \boxtimes^c H_2.$ 

\begin{defn} \label{defn: tensorproductofdieudonne}
Let $M_1,M_2$ be $\mathcal F$-profinite modules over the Dieudonné ring $\mathcal R$. Let $\alpha: \mathcal{F} \otimes_{W(k)} (M_1 \hat{\otimes}_{W(k)} M_2) \rightarrow \mathcal{R} \otimes_{\mathcal V} (M_1 \hat{\otimes}_{W(k)} M_2)$ be the natural map. Denote by $\mathcal{R} \hat{\otimes}^{\mathbb{D}}_{\mathcal V} (M_1 \hat{\otimes}_{W(k)} M_2)$ the completion of $$\mathcal{R} \otimes_{\mathcal V} (M_1 \hat{\otimes}_{W(k)} M_2)$$ along the $\mathcal F$-submodules generated by \begin{equation}\label{eq:boxproductforFprofinites} \alpha(I \otimes_{W(k)} (M_1 \hat{\otimes}_{W(k)} M_2) +  \mathcal F \otimes_{W(k)} N) \subset \mathcal F \otimes_{W(k)} (M_1 \hat{\otimes}_{W(k)} M_2) \end{equation} where $N \subset M_1 \hat{\otimes}_{W(k)} M_2$ ranges over a basis of neighborhoods of $0$ defining the topology on $M_1 \hat{\otimes}_{W(k)} M_2,$ and $I \subset \mathcal F$ ranges over all $\mathcal F$-ideals such that $\mathcal F/I$ has finite length as a $W(k)$-module. 

Then
\[
M_1 \boxtimes M_2 = \mathcal{R} \hat{\otimes}^\mathbb{D}_{\mathcal V} (M_1 \hat{\otimes}_{W(k)} M_2)/\bar K,
\]
where $\bar K$ is the closure of the Dieudonné submodule $K$ generated by $$F \otimes x \otimes Vy- 1 \otimes Fx \otimes y, F \otimes Vx \otimes y - 1 \otimes x \otimes Fy, \quad x \in M_1, y \in M_2.$$ We define $M_1 \boxtimes^c M_2$ to be the connected part of $M_1 \boxtimes M_2.$ 
\end{defn} 
\begin{remark} \label{rmk:connectedofDieudonne} 
If we restrict $I$ in \eqref{eq:boxproductforFprofinites} to the ideals $(F^n)$ for $n \geq 0$ and denote the resulting $\mathcal R$-module by $\mathcal{R} \hat{\otimes}^F_{\mathcal V} (M_1 \hat{\otimes}_{W(k)} M_2)$, then
\[
M_1 \boxtimes^c M_2 \cong \mathcal{R} \hat{\otimes}^F_{\mathcal V} (M_1 \hat{\otimes}_{W(k)} M_2)/\bar K
\]
for the same $K$, with closure taken inside the new enclosure.
\end{remark}

We see that $M_1 \boxtimes M_2$ corepresents the functor taking $K \in \DDMod{k}$ to the set of Dieudonné pairings $M_1 \times M_2 \rightarrow K,$ and it is clear that $M_1 \boxtimes^c M_2$ corepresents its restriction to connected Dieudonné modules. Given two formal connected Hopf algebras $H_1,H_2,$ we thus have a a natural map $\DieuFor(H_1) \boxtimes^c \DieuFor(H_2) \rightarrow \DieuFor(H_1 \boxtimes^c H_2)$ in $\DDMod{k}$ induced from the Dieudonné pairing $$\DieuFor(H_1) \times \DieuFor(H_2) \rightarrow \DieuFor(H_1 \boxtimes^c H_2)$$ constructed in Lemma \ref{lemma:bilineargivespairing}. 
\begin{prop} \label{prop:Dieudonnemonoidal}
Let $H_1,H_2 \in \FHopfc.$ Then the natural map $$\gamma: \DieuFor(H_1) \boxtimes^c \DieuFor(H_2) \rightarrow \DieuFor(H_1 \boxtimes^c H_2)$$ is an isomorphism.
\end{prop}
\begin{proof}
Note that $C_k$ generates the category of connected formal Hopf algebras, since the formal group scheme $CW^c_k$ is a cogenerator of the category of connected formal groups (this is used in the proof of Thm.~\ref{thm:dieudonneforformalgroups}). 
Thus $H_1$ has a presentation
\[
C_k^\alpha \rightarrow C_k^\beta \rightarrow H_1 \rightarrow 0,
\]
where $C_k^\alpha$ denotes the $\alpha$-fold sum in $\FHopfc$, which is the completed tensor product $\hat\otimes_k$, for a possibly infinite cardinal $\alpha$.
Since $\boxtimes^c$ is right exact and distributes over $\hat\otimes$, we have an exact sequence
\[
(C_k \boxtimes^c H_2)^\alpha \rightarrow (C_k \boxtimes^c H_2)^\beta \rightarrow H_1 \boxtimes^c H_2 \rightarrow 0.
\]
One then easily sees that it is enough to prove that the natural map $$\DieuFor(H_1) \boxtimes^c \DieuFor(H_2) \rightarrow \DieuFor(H_1 \boxtimes^c H_2)$$ is an isomorphism when $H_1$ is the base change to $k$ of a formal connected Hopf algebra with a Verschiebung lift, such as $C_k$. By symmetry, we can make the same assumption about $H_2$. By Thm.~\ref{thm:indecomposablesdieudonne}, $$\DieuFor(H_1 \boxtimes^c H_2) \cong \mathcal{R} \hat{\otimes}^F_{\mathcal V} (Q(H_1) \hat{\otimes}_{W(k)} Q(H_2)).$$  With notation as in Remark \ref{rmk:connectedofDieudonne}, we further see that $\DieuFor(H_1) \boxtimes^c \DieuFor(H_2)$ is isomorphic to
\[
\mathcal{R} \hat{\otimes}^F_{\mathcal V} (\mathcal{R} \hat{\otimes}^F_{\mathcal V} Q(H_1) ) \hat{\otimes}_{W(k)} (\mathcal{R} \hat{\otimes}^F_{\mathcal V} Q(H_2) )/\bar K.
\]
This Dieudonné module is isomorphic to $$\mathcal{R}  \hat{\otimes}^F_{\mathcal V} (Q(H_1) \hat{\otimes}_{W(k)} Q(H_2)),$$ by the map taking $F^i \otimes (F^j \otimes h_1) \otimes (F^k \otimes h_2)$ to $$F^{i+j+k} \otimes V^k(h_1) \otimes V^j(h_2).$$ A quick computation shows that this is precisely the map $\gamma.$  
\end{proof} 
\begin{example} \label{ex:apboxap}
Let $\alpha_{p}$ be as in Lemma~\ref{lemma:connectedtensorconnectednotconnected}. Then $\Reg{\alpha_p} \boxtimes^c \Reg{\alpha_p} \cong \Reg{\mathbb{G}^c_a},$ where $\mathbb{G}^c_a$ is the formal group defined by $$\mathbb{G}^c_a(A) = (\Nil(A),+).$$ To prove this isomorphism, we use Prop.~\ref{prop:Dieudonnemonoidal} to calculate $\DieuFor(\Reg{\alpha_p}) \boxtimes \DieuFor(\Reg{\alpha_p}).$ As is well-known and easy to show, $\DieuFor(\Reg{\alpha_p})= k,$ where $F=0=V$. Furthermore, we see that $$\DieuFor(\Reg{\alpha_p}) \boxtimes^c \DieuFor(\Reg{\alpha_p}) \cong k\pow{F}.$$ Since $$\DieuFor(\Reg{\mathbb{G}^c_a} ) = \varprojlim_n \DieuFor(\Reg{\alpha_{p^n}}) = k\pow{F},$$ the fact that the Dieudonné correspondence is an equivalence of categories gives that $$\Reg{\alpha_p} \boxtimes^c \Reg{\alpha_p} \cong \Reg{\mathbb{G}^c_a},$$ as claimed. Note that this example shows that in general, the tensor product of two finite formal Hopf algebras need not be finite.
\end{example}
We end this section by noting that by Matlis duality, $\DMod{k}$ gets a monoidal structure which we denote by $\boxast.$ For $K$, $L \in \DMod{k}$, define
\[
K \boxast L = I^c(I(K) \boxtimes I( L)).
\]
This equips the category $\DMod{k}$ with a symmetric monoidal structure with unit $CW(k).$ We denote by $K \boxast^u L$ the maximal unipotent Dieudonné submodule of $K \boxast L.$ Note that $K \boxast^u L = I^c(I(K) \boxtimes^c I(L)).$ In the following lemma, recall that if $K,L \in \DMod{k},$ then we write $K*L$ for $\Tor^1_{W(k)}(K,L).$ 
\begin{lemma} \label{lemma:boxast}
The symmetric monoidal structure $\boxast$ agrees with the one defined in the introduction.
\end{lemma}
\begin{proof}
This follows from Proposition ~\ref{prop:matlismonoidal} and the definition of $\boxtimes$ by straightforward dualization.
\end{proof}
Note that there is a canonical projection $\pi\colon K \boxast L \to K * L$ given by evaluation at $1 \in \mathcal R$. If $F$ acts as an isomorphism on $K$ then $\pi$ is in fact an isomorphism since by the second condition, $f(V^i) = (F^{-i} * V^i)f(1)$ has to hold, showing uniqueness, and
\[
(V * 1)f(r)=(pF^{-1} * 1)f(r) = (F^{-1} * 1)f(pr) = (1 * F)f(Vr),
\]
shows that the first condition holds if this is taken as a definition for $f$.

\subsection{The étale part of $\boxtimes$ for two connected formal Hopf algebras}
Let $H_1,H_2$ be connected formal Hopf algebras. To describe the étale part of $\Spf{H_1 \boxtimes H_2}$ is the same as to describe the Galois module structure of $$\Spf{(H_1 \boxtimes H_2)}(\bar{k}) = \colim_{k'}  \Hom_{\Alg_k}^c(H_1 \boxtimes H_2 , k')$$ where $k'$ ranges over the finite extensions of $k.$ The goal of this section is to find a description of $\Spf{(H_1 \boxtimes H_2)}(\bar{k})$ in terms of the Dieudonné modules of $H_1$ and $H_2.$ 
The following definition is the profinite version of \cite[Def.~4.5]{buchstaber-lazarev:dieudonne}.

\begin{defn} \label{def:bilinearpairing}
Let $H_1,H_2$ be formal connected Hopf algebras and $A$ be a finite length $k$-algebra. A $k$-linear map $\varphi\colon H_1 \hat{\otimes}_k H_2 \rightarrow A$ is a bilinear pairing if the following holds for any $a,a_1,a_2 \in H_1, b,b_1,b_2 \in H_2$ (using Sweedler notation):
\begin{itemize}
\item $ \varphi(a,b_1b_2) = \sum_{(a)} \varphi(a',b_1) \varphi(a'',b_2)$;
\item $\varphi(a_1a_2,b) = \sum_{(b)} \varphi(a_1,b')\varphi(a_2,b'')$ ;
\item $\varphi(1,b) = \epsilon_{H_2}(b)$;
\item $\varphi(a,1) = \epsilon_{H_1}(a).$ 
\end{itemize}
\end{defn}
Denote by $\Bil(H_1,H_2;A)$ the abelian group of bilinear pairings $H_1 \hat{\otimes}_k H_2 \rightarrow A$ and set $$\Bil(H_1,H_2;\bar{k}) = \colim_{k'} \Bil(H_1,H_2;k')$$ where $k'$ ranges through the finite extensions of $k$. Note that $\Bil(H_1,H_2;\bar{k})$ is a $\Gamma$-module in a natural way.

\begin{lemma} \label{lemma:bilinearrationalpoints}
Let $H_1$, $H_2$ be formal Hopf algebras. Then for any finite length $k$-algebra $A,$ there is an isomorphism $$\Bil(H_1,H_2;A) \cong \Spf{(H_1 \boxtimes H_2)}(A).$$ 
\end{lemma}
\begin{proof}
This follows readily from Remark \ref{rmk:structureoftensor}, since $H_1 \boxtimes H_2$ is the profinite completion of the symmetric algebra on $H_1 \hat{\otimes}_k H_2$ modulo the closure of the relations given in Definition \ref{def:bilinearpairing}. 
\end{proof}

Note that Lemma \ref{lemma:bilinearrationalpoints} implies that $\Bil(H_1,H_2;\bar{k}) \cong \Spf{(H_1 \boxtimes H_2)}(\bar{k}).$ 
\begin{lemma} \label{lemma:bilinearfactors}
Let $k'$ be a finite extension of $k,$ $H_1$ and $H_2$ connected formal Hopf algebras with Verschiebung $V_{1},V_{2}$, respectively. Then any bilinear pairing $\varphi\colon H_1 \hat{\otimes}_k H_2 \rightarrow k'$  factors through $H_1 /V^n_1 \hat{\otimes}_k H_2/V^m_2$ for some $m,n \geq 0,$ where we are taking the quotient in the category of Hopf algebras. 
\end{lemma}
\begin{proof}
By \cite[\textsection I.9, Théorème 1]{fontaine:groupes-divisibles}, for $i=1,\;2$, $H_i \cong k\pow{W_i}/I_i$ for some $k$-vector spaces $W_i$ and closed ideals $I_i$ contained in the augmentation ideals $(W_i)$.

By continuity, $\ker \varphi$  must contain an open ideal of $H_1 \hat{\otimes}_k H_2.$ Any ideal of $H_1 \hat{\otimes}_k H_2$ contains a subideal of the form $(W_1)^{p^n} \otimes H_2 + H_1 \otimes (W_2)^{p^m}$ for some $n,m \geq 0.$ 

By the definition of bilinear pairings,
\[
\varphi(x,V^n_{2}(y))^{p^n} = \varphi(x^{p^n},y).
\]
Thus if $x \in (W_1)$ then $\varphi(x,V^n_{2}(y))^{p^n}=0$ and hence, since $k'$ is reduced, $\varphi(x,V^n_{2}(y)) = 0.$ If on the other hand $x \in k$ then $\varphi(x,V^n_2(y)) = x \epsilon_{H_2}(V^n_2(y)) = 0$ if $y \in (W_2)$. Thus $\varphi$ vanishes on $H_1 \hat\otimes_k V_2^n(W_2)$. Symmetrically, $\varphi$ vanishes an $V_1^m(W_1) \hat\otimes_k H_2$, concluding the proof.

\end{proof} 
\begin{corollary} \label{cor:étalefactors}
Let $H_1$, $H_2$ be connected formal Hopf algebras. Then $$\Spf{(H_1 \boxtimes H_2)}(\bar{k}) = \colim_{m,n} \Spf{(H_1/V^m_{1} \boxtimes H_2/V^n_{2})}(\bar{k}).$$ 
\end{corollary}
\begin{proof}
This follows immediately by combining Lemmas~\ref{lemma:bilinearrationalpoints} and~\ref{lemma:bilinearfactors}.
\end{proof}
Using the above, we will now work toward giving a formula for $\Spf{(H_1 \boxtimes H_2)}(\bar{k}).$ 

In the category $\Hopf{k}$ of (bicommutative) Hopf algebras, the tensor product $\boxtimes_a$ exists by Thm.~\ref{thm:existence} end was studied in \cite{goerss:hopf-rings} and \cite{buchstaber-lazarev:dieudonne}. Given any Hopf algebra $H,$ profinite completion $\widehat{H}$ yields an object of $\FHopf.$ We denote by $\Reg{W^n_k}$ the Hopf algebra corepresenting the functor taking an algebra $A$ to its length $n$ Witt vectors. We let $\Reg{W^{n,c}_k}$ be the formal Hopf algebra corepresenting $W^{n,c} :=  \colim_m \ker (F_{W^n_k}^m:W^n_k \rightarrow W^n_k ).$ 
\begin{lemma} \label{lemma:completioncommbox}
Let $m,n \geq 0.$ Then $$(\Reg{W^{n}_k} \boxtimes_a \Reg{W^{m}_k})\;\widehat{}\; \cong \Reg{W^{n,c}_k} \boxtimes \Reg{W^{m,c}_k}$$ as formal Hopf algebras.
\end{lemma}
\begin{proof}
Both sides of the claimed isomorphism represent formal groups. We will show that the functors of points of the corresponding formal groups agree, and thus, that the underlying formal Hopf algebras are isomorphic. We will write $\Reg{W^n_k} \cong k[x_0, \ldots, x_{n-1}],$ and $\Reg{W^m_k} \cong k[y_0, \ldots, y_{m-1}].$ Let $A$ be a finite $k$-algebra and $\varphi \in \Hom_{\Alg_k}(\Reg{W^{n}_k} \boxtimes_a \Reg{W^{m}_k},A)$ be a morphism. 
Then $\varphi$ is represented by a bilinear pairing
\[
\varphi: \Reg{W^{n}_k} \otimes_k \Reg{W^{m}_k} \rightarrow A.
\]
For this to induce an element of  $\Hom_{\Alg_k}(\Reg{W^{n,c}_k} \boxtimes \Reg{W^{m,c}_k},A),$ we see that $\varphi$ must factor through a quotient by an ideal of the form $$(x_0, \ldots , x_{n-1})^{p^s} \otimes \Reg{W^m_k} + \Reg{W^n_k} \otimes (y_0, \ldots, y_{m-1})^{p^t}$$ for some $s,t \geq 0.$  Indeed, writing
\[
\Reg{W^{n,c}_k} \cong k\pow{x_0, \ldots , x_{n-1}},
\]
a fundamental open system of neighborhoods is given by $(x_0, \ldots,x_{n-1})^{p^s}$ for $s \geq 0$, and similarly for $\Reg{W^{m,c}_k}$. Using the relations 
\begin{align*}
\varphi(x,V_{W^m_k} y)^p = \varphi(x^{p},y) \\
 \varphi(V_{W^n_k}x,y)^p = \varphi(x,y^p),
\end{align*} and
the observation that $V^n_{W^n_k} =0, V^m_{W^m_k}=0, $ for large enough $N>0$, every monomial generator  $$x^I = x_0^{i_0} x_1^{i_1} \cdots x_{n-1}^{i_{n-1}} \in (x_0, \ldots , x_{n-1})^{p^N}$$ has some $i_k > p^m.$ Then the bilinearity relations imply that $\varphi(-,y),$ $y \in \Reg{W^{m}_k},$  vanishes on $(x_0, \ldots,x_{n-1})^{p^N}.$ Indeed, take a monomial generator $x^I = x_0^{i_0} x_1^{i_1} \cdots x_{n-1}^{i_{n-1}} \in (x_0, \ldots, x_{n-1})^{p^N}$ and assume without loss of generality that $i_0>p^m.$ Then $$\varphi(x_0^{i_0} x_1^{i_1} \cdots x_{n-1}^{i_{n-1}} , y) = \sum_j \varphi(x_0^{p^m},y_j^1) \varphi(x_0^{i_0-p^m} x_1^{i_1} \cdots x_{n-1}^{i_{n-1}},y_j^2).$$ Since $$\varphi(x_0^{p^m},y_j^1) = \varphi(x_0,V^m(y_j))^{p^m} = 0,$$ the sum vanishes. By doing the same argument for the other variable, we see that $\ker \varphi$ factors as claimed. Thus, every morphism $\varphi \in \Hom_{\Alg_{k}}(\Reg{W^{n}_k} \boxtimes^a \Reg{W^{m}_k},A)$ induces an element of $$\Hom_{\Alg_{k}}^c(\Reg{W^{n,c}_k} \boxtimes \Reg{W^{m,c}_k},A)$$ and it is easy to see that this gives an isomorphism. 
\end{proof}
We let $$\iota_n\colon \Reg{W^n_k} \rightarrow \Reg{W^n_k} \boxtimes_a \Reg{W^n_k}$$ be the mod-$p$ reduction of the map from Prop.~\ref{prop:HLiota}. After completion, using Lemma \ref{lemma:completioncommbox}, we obtain a map 
\[
\iota_n\colon \widehat{\Reg{W^n_k}} \rightarrow \Reg{W^{n,c}_k} \boxtimes \Reg{W^{n,c}_k}
\]
of formal Hopf algebras. Given two formal connected Hopf algebras $H_1,H_2,$ We now define a map
\[
\Spf{(H_1 \boxtimes H_2) }(\bar{k}) \rightarrow \colim_{m,n}\Hom_{\DDModF{k}}(\mathbb{D}^f(H_1/V^m_{H_1}) \boxtimes \mathbb{D}^f(H_2/V^n_{H_2}) , CW(\bar{k}))
\]
which will turn out to be an isomorphism of $\Gamma$-modules. By Cor.~\ref{cor:étalefactors}, any map $f \in \Spf{(H_1 \boxtimes H_2)}(\bar{k})$ factors as 
\[
f: H_1/V^n_{H_1} \boxtimes H_2/V^n_{H_2} \rightarrow k'
\]
for some $n \geq 0$ and $k'$ a finite extension of $k.$
As in \eqref{eq:muf}, we define
\begin{align*}
\mu_f\colon &\mathbb{D}^f(H_1/V^n_{H_1}) \times \mathbb{D}^f(H_2/V^n_{H_2})\\
 = &\Hom_{\FHopf}(\Reg{W^{n,c}_k},H_1/V^n_{H_1}) \times \Hom_{\FHopf}(\Reg{W^{n,c}_k},H_2/V^n_{H_2})\\
\xrightarrow{\boxtimes} & \Hom_{\FHopf}(\Reg{W^{n,c}_k} \boxtimes \Reg{W^{n,c}_k} , H_1/V^n_{H_1} \boxtimes H_2/V^n_{H_2})\\
\xrightarrow{\iota_n^*} & \Hom_{\FHopf}(\widehat{\Reg{W^n_k}} , H_1/V^n_{H_1} \boxtimes H_2/V^n_{H_2})\\
\xrightarrow{f_*} &\Hom_{\FAlg{k}}(\widehat{\Reg{W^n_k}},k') = W_n(k') \subset CW(\bar{k}).
\end{align*}
The following lemma is proved just as Lemma \ref{lemma:bilineargivespairing}.
\begin{lemma}\label{lemma:mapgivesdieudonne}
The map $\mu_f$ defines a Dieudonné pairing $$\mu_f: \DieuFor(H_1/V^n_{H_1}) \times \DieuFor(H_2/V^n_{H_2}) \rightarrow \Hom_{\FAlg{k}}(\widehat{\Reg{W^n_k}},k') = W_n(k') \subset CW(\bar{k})$$ and thus gives a map $$\mu_f: \DieuFor(H_1/V^n_{H_1}) \boxtimes \DieuFor(H_2/V^n_{H_2}) \rightarrow CW(\bar{k})$$ of Dieudonné modules.\qed
\end{lemma}
One easily shows that this map is independent of the choices of $n$ and $k'$. Given $ f \in \Hom_{\Alg_k}(H_1 \boxtimes H_2,\bar{k})$, we associate a map $\mu_f : \DieuFor(H_1/V^n_{H_1}) \boxtimes \DieuFor(H_2/V^n_{H_2}) \rightarrow CW(\bar{k}).$  This defines a morphism 
\[
t\colon \Hom_{\Alg_k}(H_1 \boxtimes H_2,\bar{k}) \rightarrow \colim_{m,n}  \colim_{k' \subset \bar{k}} \Hom_{\DMod{k}}^c(\DieuFor(H_1/V^n_{H_1}) \boxtimes \DieuFor(H_2/V^m_{H_2}), CW(k'))
\] of $\Gamma$-modules, where $\Hom_{\DMod{k}}^c$ denotes continuous morphisms of Dieudonné modules and $k'$ ranges over the finite extensions of $k$.  To simplify notation, let 
$$\Hom_{\DMod{k}}^c(-, CW(\bar{k})) = \colim_{k' \subset \bar{k}} \Hom_{\DMod{k}}^c(-, CW(k')).$$

\begin{prop} \label{prop:tgivesiso}
The map $$t\colon \Hom_{\Alg_k}(H_1 \boxtimes H_2,\bar{k}) \rightarrow \colim_{m,n}  \Hom_{\DMod{k}}^c(\DieuFor(H_1/V^n_{H_1}) \boxtimes \DieuFor(H_2/V^m_{H_2}), CW(\bar{k}))$$ is an isomorphism of $\Gamma$-modules. 
\end{prop}
We first show a special case:
\begin{lemma} \label{lemma:tforV} 
Prop.~\ref{prop:tgivesiso} holds when $H_1,H_2$ are  connected smooth formal Hopf algebras with $V_{1} = 0$ and $V_{2} = 0.$
\end{lemma}
\begin{proof}[Proof of Lemma \ref{lemma:tforV}]
Let us write $$H_1 \cong \Spf k\pow{(X_i)_{i \in S}}$$ and $$H_2 \cong \Spf k\pow{(Y_j)_{j \in T}}$$ for some index sets $S$ and $T$ such that each $X_i$ and each $Y_j$ is primitive. It is clear that a bilinear pairing $$\varphi: H_1 \hat{\otimes}_k H_2 \rightarrow k'$$ for $k'$ a finite extension of $k,$ is determined by what it does on elements of the form $X_i \otimes Y_j,$  for $m,n \geq 0,$ $i \in S$ and $j \in T.$ It is then clear that giving a pairing $\varphi$ as above is equivalent to giving a continuous map $Q(H_1) \hat{\otimes}_k Q(H_2) \rightarrow k'$. On the other hand, to give a Dieudonné pairing $$\psi: \DieuFor(H_1) \otimes \DieuFor(H_2) \rightarrow k',$$ we must have that $\psi(F(x),y) = \psi(x,V(y))^p = 0,$ for $x \in \DieuFor(H_1), y \in \DieuFor(H_2).$ Doing this for the other variable, we see that this means that $\psi$ factors through $$\dfrac{\DieuFor(H_1)}{F \DieuFor(H_1)} \hat{\otimes}_{W(k)}  \dfrac{\DieuFor(H_2)}{F\DieuFor(H_2)} \cong Q(H_1) \hat{\otimes}_k Q(H_2),$$ that is, giving a Dieudonné pairing $\DieuFor(H_1) \times \DieuFor(H_2) \rightarrow k'$ is equivalent to giving a continuous map $$Q(G_1) \hat{\otimes}_k Q(G_2) \rightarrow k'.$$ What remains is to show that the map $$t: H_1 \boxtimes H_2(\bar{k}) \rightarrow \Hom_{\DMod{k}}^c(\DieuFor(H_1) \boxtimes \DieuFor(H_2), \bar{k}),$$ is an isomorphism. This follows easily from noting that the map $$\iota_1: \Reg{W^1_k} \rightarrow \Reg{W^{1,c}_k} \boxtimes \Reg{W^{1,c}_k}$$ takes $x \in \Reg{W^1_k}$ to $x \boxtimes x.$ 
\end{proof}
\begin{proof}[Proof of Prop.~\ref{prop:tgivesiso}]
We will prove this by induction on $m$ and $n,$ following the lines of \cite[Prop.~1.3.28]{hopkins-lurie:ambidexterity}. We assume that $V^n_{H_1} =0$ and that $V^m_{H_2} = 0$ for some $m,n \geq 0.$  As in the proof of Prop.~\ref{prop:Dieudonnemonoidal}, one can assume that $\Spf{H_1} = \prod_{i \in I} W^{n,c}_k$ and $\Spf{H_2} = \prod_{j \in J} W^{m,c}_k$ for some index sets $I$ and $J$ respectively. Note that 
\[
H_1 = \widehat\bigotimes_{i \in I} \Reg{W^{n,c}_k}, \quad H_2 = \widehat\bigotimes_{j \in J} \Reg{W^{m,c}_k}
\] in the category of formal Hopf algebras over $k.$ We abbreviate
\[
 A_{n,m} = \Spf{(\widehat\bigotimes_{i \in I} \Reg{W^{n,c}_k} \boxtimes \widehat\bigotimes_{j \in J} \Reg{W^{n,c}_k})}(\bar{k}).
\]
By considering the exact sequence of formal groups
\[
0 \rightarrow \prod_{j \in J} W^{m-1,c}_k \xrightarrow{\prod V}  \prod_{j \in J} W^{m,c}_k \rightarrow \prod_{j \in J } W^{1,c}_k \rightarrow 0
\]
where the first map is the product of the Verschiebung maps $V\colon W^{m-1,c}_k \rightarrow W^{m,c}_k,$ we get an exact sequence 
\begin{equation}\label{eq:exactsequenceoftensorproducts}
\begin{tikzcd} 0 \arrow[r] &  A_{n,m-1} \arrow[r,"1 \boxtimes \prod V"]  & A_{n,m} \arrow[r] & A_{n,1}.\end{tikzcd}
\end{equation}
We need to show this sequence is also exact at the right. To see this, note that the subcoalgebra
\[
H_1'(n) \otimes_k H_2'(m) =_{\operatorname{def}} (\bigotimes_{i \in I} \Reg{W^n_k}) \otimes_k (\bigotimes_{j \in J} \Reg{W^{m}_k}) \subseteq (\widehat{\bigotimes_{i \in I}} \Reg{W^{n,c}_k} )\hat{\otimes}_k (\widehat{\bigotimes_{j \in J}} \Reg{W^{m,c}_k})
\]
is dense. By \cite[Lemma 7.1]{goerss:hopf-rings}\footnote{Goerss works with graded Hopf algebras, but his methods of proof goes through in the ungraded case, as explained in \cite{buchstaber-lazarev:dieudonne}},
\[
H_1'(n) \boxtimes_a H_2'(m) \cong \Sym\left((\bigoplus_{i \in I} Q(\Reg{W^n_k})) \otimes_k (\bigoplus_{j \in J} Q(\Reg{W^{m}_k}))\right).
\]
Thus to give a bilinear pairing $H_1'(n) \otimes_k H_2'(m) \to k'$ is equivalent to specifying where to send the indecomposables. It follows that every bilinear pairing $\overline{\varphi}\colon  H_1'(n) \otimes H_2'(1) \to k'$ can be lifted to a bilinear pairing $\varphi\colon H_1'(n) \otimes H_2'(m) \to k'$. Using an argument as in the proof of Lemma~\ref{lemma:completioncommbox}, all of these bilinear pairings are continuous. Thus the sequence \eqref{eq:exactsequenceoftensorproducts} is right exact. 

With the notation
\[
B_{n,m} =  \Hom^c_{\DMod{k}}(\DieuFor(H_1) \boxtimes \DieuFor(H_2),CW(\bar{k})),
\]
we then have an exact sequence 
\[
\begin{tikzcd} 0 \arrow[r] & B_{n,m-1} \arrow[r,"1 \boxtimes \Pi V"] &  B_{n,m} \arrow[r] &  B_{n,1} .\end{tikzcd}
\]
Together with \eqref{eq:exactsequenceoftensorproducts}, this gives rise to a commutative diagram of exact sequences
\[
\begin{tikzcd}0 \arrow[r] & A_{n,m-1} \arrow[d] \arrow[r,"1 \boxtimes \Pi V"] & A_{n,m} \arrow[r] \arrow[d] &  A_{n,1} \arrow[d] \arrow[r] &  0  \\   0 \arrow[r] & B_{n,m-1}  \arrow[r,,"1 \boxtimes \Pi V"] &  B_{n,m} \arrow[r] &  B_{n,1}, & \end{tikzcd}
\]
and by induction and the snake lemma, we are reduced to the case $m=1.$ By another induction, it is enough to show that $A_{1,1} \rightarrow  B_{1,1}$ is an isomorphism. But in this case, the statement follows from Lemma~\ref{lemma:tforV}.
\end{proof}

\begin{corollary} \label{cor:etaleofconn}
Let $H_1$ and $H_2$ be connected formal Hopf algebras over $k.$ Then the étale part of $\Spf{(H_1 \boxtimes H_2)}$ is the étale formal group associated with the $\Gamma$-module
\[
\colim_{m,n} \Hom_{\DMod{k}}(\DieuFor(H_1/V^n_{H_1}) \boxtimes \DieuFor(H_2/V^n_{H_2}) , CW(\bar{k})).
\]\qed
\end{corollary}

\begin{example} \label{ex:etaleofap}
Let $\alpha_p$ be as in Lemma~\ref{lemma:connectedtensorconnectednotconnected}. As we saw there, $\alpha_p \boxtimes \alpha_p$ is not connected. To compute its étale part, we must  compute the Dieudonné pairings $$\DieuFor(\alpha_p) \boxtimes \DieuFor(\alpha_p) \rightarrow CW(\bar{k}).$$ We know that $\DieuFor(\alpha_p) = k,$ so a Dieudonné pairing, being bilinear, is determined by where it sends $$(1,1) \in k \times k \cong \DieuFor(\alpha_p) \times \DieuFor(\alpha_p).$$ We see that to be a Dieudonné pairing, $(1,1)$ must be mapped into a submodule of $CW(\bar{k})$ where $V=0,$ i.e into $\bar{k}.$ This argument shows that the $\Gamma$-module of Dieudonné pairings $\DieuFor(\alpha_p) \boxtimes \DieuFor(\alpha_p) \rightarrow CW(\bar{k})$ is isomorphic to $\bar{k},$ as expected.  
\end{example}

\begin{remark}\label{rmk:dieufornotmonoidal}
A natural hope would be for the functor $\DieuFor\colon \FHopfc \rightarrow \DDModF{k}$ to be monoidal, i.e. that for $H_1, H_2$ two formal connected Hopf algebras, we would have $\DieuFor(H_1 \boxtimes H_2) \cong \DieuFor(H_1) \boxtimes \DieuFor(H_2).$ Unfortunately, this is not true, we only have $\DieuFor(H_1 \boxtimes^c H_2) \cong \DieuFor(H_1) \boxtimes^c \DieuFor(H_2).$ To see this, suppose that $\DieuFor$ was monoidal with respect to $\boxtimes,$ let $H_1 = H_2 = \Reg{CW^c_k}$ and assume for simplicity that $k$ is algebraically closed. Then 
\[
\Hom_{\FGps{k}}(\underline{\smash{\Z/p\Z}} , \Spf{\Reg{CW^c_k} \boxtimes \Reg{CW^c_k}}) \cong \Hom_{\DDModF{k}}(\DieuFor(\Reg{CW^c_k}) \boxtimes \DieuFor(\Reg{CW^c_k}) , k)
\]
where $V(x)=0$ and $F(x)=x^p$ on $k.$ By the universal property of $\boxtimes,$ coupled with Thm.~\ref{thm:indecomposablesdieudonne}, this is the same as $V$-equivariant maps $Q(CW^c_{W(k)}) \hat{\otimes}_{W(k)} Q(CW^c_{W(k)}) \rightarrow k.$ An elementary calculation shows then that $$ \Hom_{\DDModF{k}}(\DieuFor(\Reg{CW^c_k}) \boxtimes \DieuFor(\Reg{CW^c_k}) , k) \cong \prod_{i=-\infty}^\infty k.$$

However, note that $$\Hom_{\FGps{k}}(\underline{\smash{\Z/p\Z}}, \Spf{\Reg{CW^c_k} \boxtimes \Reg{CW^c_k}})$$ is isomorphic to the $p$-torsion of $\Spf{\Reg{CW^c_k} \boxtimes \Reg{CW^c_k}}(k).$ Using Cor.~\ref{cor:etaleofconn}, one shows that this is isomorphic to $\bigoplus_{i=-\infty}^\infty k,$ which is not isomorphic to $\prod_{i=-\infty}^\infty k.$ 
\end{remark}

\section{Tensor products of affine abelian groups schemes in positive characteristic} \label{sec:affine-groups}

Let $k$ be a perfect field of characteristic $p>0.$ By Thm.~\ref{thm:existence}, the tensor product of any two affine groups exists. In this section we first use the results from Section \ref{sec:tensorprodHopf} to give a formula for the Dieudonné module of the tensor product of two unipotent affine groups.  We then use this result to competely describe the tensor product of two affine groups. Note that by Thm.~\ref{thm:dieudonneforpadic}, the category of unipotent groups over $k$ is equivalent to $ \DModV{k},$ those Dieudonné modules that are filtered colimits of its Verschiebung kernels. By Cartier duality, the category of unipotent groups is anti-equivalent to the category of connected formal groups. 
Suppose that we have a bilinear map $f\colon G_1 \times G_2 \rightarrow G_3$ of affine groups. By Cartier duality, we get a cobilinear map $G_3^* \rightarrow G_1^* \times G_2^* $ of formal groups, and by looking at the representing objects, we get a bilinear map of formal Hopf algebras $$\Reg{G_1^*} \hat{\otimes}_k \Reg{G_2^*} \rightarrow \Reg{G_3^*}.$$ Conversely, given a cobilinear map $H_3 \rightarrow H_1 \times H_2$ of formal groups, corresponding to a bilinear map $\Reg{H_1} \hat{\otimes} \Reg{H_2} \rightarrow \Reg{H_3}$ of formal Hopf algebras, Cartier duality produces a bilinear map $$H_1^* \times H_2^* \rightarrow H_3^*$$ of affine groups. This shows that there is a bijective correspondence between bilinear maps of formal Hopf algebras and bilinear maps of affine groups. We thus have:

\begin{thm}\label{thm:DieudonneCovariant}
Let $G_1,G_2$ be unipotent groups over a perfect field $k$ of positive characteristic. Then the connected part of $(G_1 \otimes G_2)^*$ is isomorphic to the formal connected group scheme associated to $$\DieuFor(G_1^*) \boxtimes^c \DieuFor(G_2^*),$$ while the étale part is isomorphic to the étale formal group scheme associated to $$\colim_{m,n} \Hom_{\DDModF{k}}(\DieuFor(G_1^*[V^n]) \boxtimes \DieuFor(G_2^*[V^m]) , CW(\bar{k}))$$
\end{thm}
\begin{proof}
From our discussion above, the functor taking an affine group $G$ to the formal Hopf algebra $\Reg{G^*}$ corepresenting its Cartier dual is monoidal. Thus, $\Reg{(G_1 \otimes G_2)^*} \cong \Reg{G_1^*} \boxtimes \Reg{G_2^*},$ where $^*$ denotes the Cartier dual. From Prop.~\ref{prop:Dieudonnemonoidal}, we get that the Dieudonné module of the connected part of $\Reg{G_1^* \boxtimes G_2^*}$ is given by $$\DieuFor(\Reg{G_1^*}) \boxtimes^c D(\Reg{G_2^*}) \cong \DieuFor(G_1^*) \boxtimes^c \DieuFor(G_2^*),$$ and the statement about the connected part follows. The étale part follows from Cor.~\ref{cor:etaleofconn}.
\end{proof}
We now give the description of the unipotent part of the Dieudonné module of $G_1 \otimes G_2,$ without first dualizing.
\begin{thm}\label{thm:DieudonneContravariant}
Let $G_1,G_2$ be unipotent groups over a perfect field $k$ of positive characteristic. Then the unipotent part of $D(G_1 \otimes G_2)$ is isomorphic to $$D(G_1) \boxast^u D(G_2).$$ 
\end{thm}
\begin{proof}
By Thm.~\ref{thm:DieudonneCovariant} and Prop.~\ref{prop:DieudonneCommuteCartierUni}, the connected part of $\DieuFor((G_1 \otimes G_2)^* )$ is isomorphic to  $$\DieuFor(G_1^*) \boxtimes^c \DieuFor(G_2^*) \cong I(D(G_1)) \boxtimes^c I(D(G_2)).$$ Since $D(G_1) \boxast^u D(G_2) = I^c (I(D(G_1)) \boxtimes^c I(D(G_2))),$  another application of Prop.~\ref{prop:DieudonneCommuteCartierUni} together with the monoidality of $I$ and $I^c$  gives the claimed result.
\end{proof}

Together with Prop.~\ref{prop:multiplicativetensor}, we conclude:
\begin{corollary} \label{cor:tensorcharp}
Let $G_1,G_2$ be affine groups over a perfect field of characteristic $p>0$ and split $G_i = G_i^u \times G_i^m, i=1,2$ into sums of a unipotent group and a multiplicative group. Denote by $M_i=\Gr(G_{i,\bar k}^m)$ the $\Gamma$-module associated with $G_i^m$. Then $G_1 \otimes G_2$ is the group whose unipotent part is characterized by $D(G_1^u \otimes G_2^u) \cong D(G_1^u) \boxast D(G_2^u)$ and whose multiplicative part is the multiplicative group scheme associated to the $\Gamma$-module
\begin{align*}
&\Z[1/p] \otimes M_1 * M_2\\
\oplus& \Hom^c_{\Ab}(\hat\pi_0(G_1^u)(\bar{k}),\Gr(G_{2,\bar{k}}))  \oplus \Hom^c_{\Mod_{\Ab}}(\hat\pi_0(G_2^u)(\bar{k}),\Gr(G_{1,\bar{k}}))  \\  
\oplus& \colim_{m,n} \Hom^c_{\DDMod{k}}(\DieuFor((G_1^u)^*[V^n]) \boxtimes \DieuFor((G_2^u)^*[V^m]) , CW(\bar{k})), 
\end{align*}
where $\Gamma$ acts on each Hom by conjugation and $[V^m]$ denotes the kernel of the $m$th Verschiebung. \qed
\end{corollary}
\printbibliography

\end{document}